\documentclass{amsart}
\usepackage[backend=biber]{biblatex}
\usepackage{topproc}
\usepackage{algorithm}
\usepackage{algorithmic}
\usepackage{graphicx,xcolor}
\usepackage{geometry} % 添加 geometry 包
\usepackage{tikz-cd}
\author{Wei Gong}
\author{Yuanda Ye}
\address{SKLMS, Institute of Computational Mathematics and Scientific/Engineering Computing, Academy of Mathematics and Systems Science, Chinese Academy of Sciences, Beijing, China.}
\email{wgong@lsec.cc.ac.cn}
\address{University of Chinese Academy  of  Sciences \& Institute of Computational Mathematics and Scientific/Engineering Computing, Academy of Mathematics and Systems Science, Chinese Academy of Sciences, Beijing, China.}
\email{yeyuanda@lsec.ac.cc.cn}
\thanks{The authors were supported by the Strategic Priority Research Program of the Chinese Academy of Sciences XDB0640000 \& XDB0640200, the National Key Research and Development Program of China (2022YFA1004402), the NSFC under grant no. 12494543 and 12471393.}

\title[A Penalty Method for Non-Self-Adjoint Topology Optimization]%
{A Penalty Method for Non-Self-Adjoint Topology Optimization}
\begin{document}

%\topProcLogo

\begin{abstract}
We propose a novel penalty method framework for the non-self-adjoint topology optimization problems, taking compliant mechanism problems as an example, by incorporating a convex nonlocal perimeter approximation scheme. We rigorously analyze the existence of solutions to the optimization problem derived from the penalty method. Furthermore, we establish that the discrete problem \(\Gamma\)-converges to the continuous problem, ensuring consistency across scales. To solve the discrete problem, we develop a projected gradient method that guarantees strict monotonic descent of the objective function. We also extend the framework to the heat dissipation problem and propose a generalized material interpolation function (GMIF), which allows for a targeted control of the topological connectivity in the resulting optimal design. Numerical experiments on the compliant mechanism and heat dissipation problems validate the effectiveness of the proposed method. This framework provides a robust approach to addressing complex optimization challenges in computational mathematics with potential applications in engineering design.
\end{abstract}

\maketitle
\section{Introduction}
Topology optimization has emerged as a powerful design tool for determining the optimal distribution of material within a prescribed domain, with applications that span structural mechanics, fluid dynamics, heat transfer, and multiphysics systems \cite{WANG2003227,Bend1989,Sigmund2013,ZHU2020103622}. Unlike shape optimization, which operates on fixed topological configurations, topology optimization allows the nucleation of new holes and the merging of existing boundaries, thus offering significantly greater design flexibility \cite{alla,Sigmund2013}. Since the seminal work of Bendsøe and Kikuchi \cite{Bend1989} on the homogenization approach, numerous methods have been developed, including the solid isotropic material with penalization (SIMP) method \cite{Bend1989}, level set methods \cite{OSHER198812,Stolpe2001,Osher2003,ALLAIRE2004363}, phase field approaches \cite{LI2025103939,JIN2024112932}, evolutionary structural optimization (ESO) \cite{Xie1993}, and more recently thresholding dynamics methods \cite{CHEN2024113119,CHEN202311891216}.

Among the various problem classes in topology optimization, non-self-adjoint problems present distinctive mathematical and computational challenges. In contrast to self-adjoint problems such as minimum compliance design—where the state and adjoint equations coincide—non-self-adjoint problems require solving separate state and adjoint equations, with the objective functional typically depending on both solutions in a non-symmetric manner. Prototypical examples include compliant mechanism design, where the goal is to maximize the output displacement at a specific location subject to an input force, and heat transfer optimization with design-dependent heat generation rates. These problems are characterized by their non-convexity, the potential existence of multiple local minima, and the sensitivity of solutions to numerical parameters and initial guesses.
The numerical treatment of non-self-adjoint topology optimization problems has traditionally relied on gradient-based optimization methods, such as the method of moving asymptotes (MMA) or the optimality criteria (OC) methods \cite{alla}. Although effective in many scenarios, these approaches often suffer from slow convergence, sensitivity to algorithmic parameters, and the need for careful continuation strategies to avoid premature convergence to suboptimal designs \cite{Deaton2014}. Furthermore, the rigorous mathematical analysis of the convergence properties for these methods remains limited (\cite{Suli2022}), particularly with respect to the relationship between discrete and continuous formulations.

In recent years, thresholding dynamics methods have attracted considerable attention as an alternative paradigm for topology optimization. Originally developed by Merriman, Bence, and Osher (MBO) for curvature-driven interface motion \cite{Ruuth2001}, these methods have been successfully extended to image segmentation \cite{Shen2002,ESEDOGLU2006367}, wetting dynamics \cite{XU2017510}, and topology optimization for fluids \cite{CHEN2022} and structures \cite{CHEN202311891216}. The key advantages of the thresholding dynamics method include its simplicity of implementation, high computational efficiency, the natural handling of topological changes through convolution and thresholding operations, and, most importantly, the monotonic descent of the objective functional that is independent of parameters related to the perimeter approximation. However, the direct application of the thresholding dynamics method to non-self-adjoint problems is nontrivial, as the standard energy decay property may fail when the constraint set varies with the design variables \cite{CHEN2024113119}. To address these limitations, Chen et al. \cite{CHEN2024113119} recently proposed a prediction-correction based iterative convolution-thresholding method (ICTM) for heat transfer problems with design-dependent heat generation. Their approach explicitly enforces monotonic descent of the objective functional through a correction step that adjusts the predicted material update based on the actual energy evaluation. Although effective, this method requires solving additional partial differential equations during the correction phase, which increases the computational cost. Moreover, the theoretical foundation of the prediction-correction strategy, particularly with regard to convergence to the continuous problem, remains to be fully established.

In this paper, we propose a penalty-based reformulation tailored for non-self-adjoint topology optimization problems, focusing on compliant mechanism problems. Our main contributions are as follows.
\begin{enumerate}
    \item Transformation of the original problem into an equivalent bilevel formulation in terms of stress variables, followed by a single-level penalized form that enforces the state constraint via a differentiable penalty term. We prove equivalence for sufficiently large penalty parameters.
    
    \item A rigorous convergence analysis using \(\Gamma\)-convergence, showing that the discrete penalized problems \(\Gamma\)-converge to the continuous one and that local minimizers of the discrete problems converge to isolated local minimizers of the continuous problem. 
    
    \item The proposed generalized material interpolation function (GMIF) allows for a  targeted control of the topological connectivity in the resulting optimal design. This is particularly important for heat transfer problems, as thickness constraints are usually imposed for the purpose of  manufacturability. 
\end{enumerate}
We further establish an abstract theoretical framework for a family of functions controlling the shape connectivity and extend the approach to the heat dissipation problems.

The remainder of the paper is organized as follows. Section~\ref{sec:problem} introduces the compliant mechanism problem, its penalty-based reformulation, and the proposed numerical algorithm. Section~\ref{sec:theory} presents the theoretical results, including equivalence proofs and \(\Gamma\)-convergence analysis. Section~\ref{sec:heat} extends the proposed penalty method to the heat transfer problem, details the numerical algorithm, and introduces the family of penalized problems. Section~\ref{sec:numerics} presents the numerical results. Finally, Section~\ref{sec:conclusion} concludes the article and discusses future directions.
  
 \section{A penalty method for compliant mechanism problems}\label{sec:problem}
We first propose an effective convergent algorithm for the compliant mechanism problem. The general idea can be described as follows: 
\begin{enumerate}
    \item Reformulate the compliant mechanism problem to a penalized double minimization problem;
    
    \item Analyze the equivalence between the penalized double minimization problem and the original optimization problem;
    
    \item Propose an efficient algorithm for the penalized optimization problem.
\end{enumerate}
In the derivation of the first part, we omit the perimeter constraint, as it does not affect the derivation.

\subsection{Reformulation into a penalized double minimization problem}
Let $\Omega\subset\mathbb{R}^d$ be a bounded Lipschitz domain and $\partial\Omega=\bar{\Gamma}_D\cup \bar{\Gamma}_N$ such that $\Gamma_D\cap \Gamma_N=\emptyset$. 
For given boundary forces $q_{\rm in},q_{\rm out}\in L^2(\Gamma_N)$, consider the following topology optimization problem
\begin{align}\label{opt_obj}
    &\min_{(\chi,u)\in \mathcal{U}\times H_D^1(\Omega)^d}\ J(\chi,u) = \int_{\Gamma_N}q_{\rm out}\cdot u\ ds
\end{align}
subject to
\begin{equation}\label{elasticity_state}                                
\int_{\Omega}E(\chi)\varepsilon(u):\varepsilon(\hat{u})\ dx=\int_{\Gamma_N}q_{\rm in}\cdot \hat{u}\ ds\quad \forall \hat{u}\in H_D^1(\Omega)^d:=\{v\in H^1(\Omega)^d:\ v=0\ \mbox{on}\ \Gamma_D\},
\end{equation}
where $E(\chi):=(\chi E_{\max} +(1-\chi)E_{\min})E_0$ for some constant tensor $E_0$, positive constants $E_{\max}$ and $E_{\min}$. For some $\beta\in (0,1)$ the set of admissible densities is given by  
\begin{equation}
 \mathcal{U}: =\Big\{\chi\in BV(\Omega):\ \chi\in \{0,1\},\ \int_{\Omega}\chi\ dx \leq\beta|\Omega|\Big\}.\nonumber   
\end{equation}
Here, $u$ is the displacement of the structure under external force $q_{\rm in}$ and $\chi$ is the distribution of the material. We want to find the optimal distribution of materials to minimize compliance under another external force $q_{\rm out}$. For simplicity, we assume that $q_{\rm in},q_{\rm out}$ are located on the same boundary $\Gamma_N$, but can  be defined in different parts of the domain boundary.

The problem \eqref{opt_obj}-\eqref{elasticity_state} is a classical PDE-constrained optimization problem, to compute its derivative, we use the well-known adjoint approach.  Introduce the adjoint problem: find $v\in H_D^1(\Omega)^d$ such that
\begin{equation}\label{elasticity_adjoint}
    \int_{\Omega}E(\chi)\varepsilon(v):\varepsilon(\hat{v})\ dx=\int_{\Gamma_N}q_{\rm out}\cdot \hat{v}\ ds\quad \forall \hat{v}\in H_D^1(\Omega)^d.
\end{equation}
For any $u,v\in H_D^1(\Omega)^d$ we define the following bilinear and linear forms:
\begin{align}
    &a(\chi;u,v): = \int_{\Omega}E(\chi)\varepsilon(u):\varepsilon(v)\ dx,\\
    &\ell_{\rm out}({v}): = \int_{\Gamma_N}q_{\rm out}\cdot {v}\ ds,\quad \ell_{\rm in}({v}): = \int_{\Gamma_N}q_{\rm in}\cdot {v}\ ds.
\end{align}
Setting $\hat u=v$ in \eqref{elasticity_state} and $\hat v=u$ in \eqref{elasticity_adjoint}, we have
\begin{eqnarray}
\ell_{\rm in}(v)=a(\chi;u,v)=\ell_{\rm out}(u). \label{equivalence}   
\end{eqnarray}

Therefore, the optimization problem \eqref{opt_obj}-\eqref{elasticity_state} can be transformed into the following form:
\begin{equation}\label{orign_problem}
    \min_{(\chi,u)\in \mathcal{U}\times H_D^1(\Omega)^d}J(\chi,u)=\ell_{\rm out}(u)\quad \mbox{subject\ to}\quad a(\chi;u,\hat{u})=\ell_{\rm in}(\hat{u})\quad\forall\hat{u}\in H_D^1(\Omega)^d.
\end{equation}
Note that the constraint equation of the above optimization problem is the Euler-Lagrange equation of the following optimization problem:
\begin{eqnarray}
 u=\arg\min\limits_{w\in H_D^1(\Omega)^d} \quad   \dfrac{1}{2}a(\chi;w,w)-\ell_{\rm in}(w).\nonumber 
\end{eqnarray}
If we define 
\begin{align*}
    g_{\rm in}(\chi,w):=\dfrac{1}{2}a(\chi;w,w)-\ell_{\rm in}(w),\quad 
    g_{\rm out}(\chi,w):=\dfrac{1}{2}a(\chi;w,w)-\ell_{\rm out}(w),
\end{align*}
then it is obvious that $ -\min\limits_{w\in H_D^1(\Omega)^d}g_{\rm in}(\chi,w) = \dfrac{1}{2}a(\chi;u,u)$ and $ -\min\limits_{w\in H_D^1(\Omega)^d}g_{\rm out}(\chi,w) = \dfrac{1}{2}a(\chi;v,v)$ with $u$ and $v$ solving the state equation \eqref{elasticity_state} and the adjoint equation \eqref{elasticity_adjoint}, respectively. 

Let $v$ be the solution to the adjoint state equation \eqref{elasticity_adjoint}. To reformulate the optimization problem, we introduce the following new stress variables
\begin{eqnarray}
\sigma=E(\chi)\varepsilon(u), \quad\rho=E(\chi)\varepsilon(v).\nonumber
\end{eqnarray}
In view of \eqref{equivalence}, we can transform the topology optimization problem into the following form \cite{CHEN202311891216}: 
\begin{align}\label{bilevel_modify}
\begin{split}
    \min\limits_{(\chi,\sigma,\varrho)\in \mathcal{U}\times S}\ \tilde{J}(\chi,\sigma,\varrho)=\int_{\Omega}E^{-1}(\chi)\sigma:\varrho\ dx\\
    \mbox{subject\ to}\quad (\sigma,\varrho)=\arg\min_{(\hat\sigma,\hat\varrho)\in S}\Big(\int_{\Omega}E^{-1}(\chi)(\hat\sigma:\hat\sigma+\hat\varrho:\hat\varrho)\ dx\Big),
\end{split}
\end{align}
where $E^{-1}(\chi):=\Big(\Big(\dfrac{1}{E_{\max}}-\dfrac{1}{E_{\min}}\Big)\chi+\dfrac{1}{E_{\min}}\Big)E_0^{-1}$, $S$ is defined by
\begin{equation}
    S:=\{(\sigma,\varrho)\in (L^2(\Omega)^{d\times d})^2:(\sigma,\varrho)_{i,j}=(\sigma,\varrho)_{j,i},\ \nabla\cdot(\sigma,\varrho)=(0,0),\ (\sigma,\varrho)\cdot\vec{n}|_{\Gamma_N}=(q_{\rm in},q_{\rm out})\}.\nonumber
\end{equation}
In Section 2.2 we will prove that the reformulated problem \eqref{bilevel_modify} is equivalent to the original problem \eqref{orign_problem}.

We utilize the penalty method for bilevel optimization problems that are well-known in optimization theory, cf. \cite{ChenMa2025}, to formulate a new optimization problem
\begin{equation}\label{penalty_formule}
    \min_{(\chi,\sigma,\varrho)\in \mathcal{U}\times S}   L(\chi,\sigma,\varrho):=\int_{\Omega}E^{-1}(\chi)\sigma:\varrho\ dx + \lambda\Big( \tilde{g}(\chi,\sigma,\varrho)-\min_{(\hat{\sigma},\hat{\varrho})\in S} \tilde{g}(\chi,\hat{\sigma},\hat{\varrho})\Big),
\end{equation}
where \(\lambda>0\) is the penalty parameter, $\tilde{g}$ is defined by
$$\tilde{g}(\chi,\sigma,\varrho):=\int_{\Omega}E^{-1}(\chi)(\sigma:\sigma+\varrho:\varrho)\ dx.$$
We remark that $\chi$, $\sigma$ and $\rho$ are independent optimization variables in the optimization problem (\ref{penalty_formule}), this is in contrast to the optimization problem \eqref{bilevel_modify}. In Section 2.2 we will prove that in this specific setting, as long as \(\lambda\) is sufficiently large, the optimization problem in penalty form \eqref{penalty_formule} is equivalent to the optimization problem \eqref{bilevel_modify}.

Of course, we could also derive an alternative penalty form, starting from the original problem involving the displacement variables
\begin{equation}\label{orign_penalty}
    \min_{(\chi,u,v)\in\mathcal{U}\times [H_D^1(\Omega)^d]^2}L(\chi)=\ell_{\rm out}(u)+\lambda\Big (g_{\rm in}(\chi,u)+g_{\rm out}(\chi,v)-\min_{(\hat u,\hat v)\in[H_D^1(\Omega)^d]^2}(g_{\rm in}(\chi,\hat{u})+g_{\rm out}(\chi,\hat{v}))\Big).
\end{equation}
However, this penalty form \eqref{orign_penalty} does not possess the property of being equivalent to the original problem \eqref{orign_problem} when \(\lambda\) is sufficiently large. Therefore, we will not discuss it in depth here. However, this form has its advantages, as it can handle cases with design-dependent body forces. In our numerical experiments, we implemented the penalty method based on the above displacement formulation. We can observe similar optimal shapes with those obtained by  the  penalty formulation based on the stress variables if the penalty parameter $\lambda$ is large enough. %which I have not seen addressed in the review articles on engineering problems that I consulted \cite{ZHU2020103622}.

\subsection{Equivalence between the original problem and the penalty form}

First, we prove the following lemma, which reveals the equivalence between the minimization problem we introduced and the linear elasticity equation.
\begin{lemma}\label{stress_equ}
    The first-order optimality condition for the minimization problem 
    \begin{equation}\label{stress_problem}
        \min_{\sigma\in S_1}\int_{\Omega}E^{-1}(\chi)\sigma:\sigma\ dx
    \end{equation}
    can be characterized as follows
    \begin{equation}
    \left\{
    \begin{aligned}
    &\sigma = E(\chi)\varepsilon(\lambda) \quad &\text{in } \Omega, \\
    &\nabla \cdot \sigma = 0 \quad &\text{in } \Omega, \\
    &\lambda = 0 \quad &\text{on } \Gamma_D, \\
    %&\lambda = -\mu \quad &\text{on } \Gamma_N, \\
    &\lambda = -\mu,\quad \sigma \cdot \vec{n} = q_{\rm in} \quad &\text{on } \Gamma_N,
    \end{aligned}
    \right.
    \end{equation}
    where
    $$ S_1:=\{\sigma\in L^2(\Omega)^{d\times d}:\sigma_{i,j}=\sigma_{j,i},\ \nabla\cdot\sigma=0,\ \sigma\cdot\vec{n}|_{\Gamma_N}=q_{\rm in}\}.$$
\end{lemma}
\begin{proof}
    First, we introduce the function space 
    $$ S_0:= \{\sigma\in L^2(\Omega)^{d\times d}:\sigma_{i,j}=\sigma_{j,i}\}$$ 
    and construct the following Lagrange functional $ L:S_0\times H^1(\Omega)\times H^1(\Omega)\rightarrow \mathbb{R}$
    $$\mathcal{L}(\sigma,\lambda,\mu) = \dfrac{1}{2}\int_{\Omega}E^{-1}(\chi)\sigma:\sigma\ dx+\int_{\Omega}(\nabla\cdot\sigma)\cdot\lambda\ dx+\int_{\Gamma_N}(\sigma\cdot\vec{n}-q_{\rm in})\cdot\mu\ ds.$$
    Then we can derive the first-order optimality condition as follows: 
    \begin{align}
        &\dfrac{\partial \mathcal{L}}{\partial \lambda}(\hat{\lambda})=\int_{\Omega}\nabla\cdot\sigma \cdot\hat{\lambda}\ dx=0,\nonumber\\
        &\dfrac{\partial \mathcal{L}}{\partial \mu}(\hat{\mu})=\int_{\Gamma_N}(\sigma\cdot\vec{n}-q_{\rm in})\cdot\hat{\mu}\ ds=0,\nonumber\\
        &\dfrac{\partial \mathcal{L}}{\partial \sigma}(\hat{\sigma})=\int_{\Omega}E^{-1}(\chi)\sigma:\hat{\sigma}\ dx+\int_{\Omega}(\nabla\cdot\hat{\sigma})\cdot\lambda\ dx+\int_{\Gamma_N}\hat{\sigma}\cdot\vec{n}\cdot\mu\ ds=0,\nonumber
    \end{align}
    where $\hat{\lambda}\in H^1(\Omega)$, $\hat{\mu}\in H^1(\Omega)$ and $\hat{\sigma}\in S_0$. 
    For the last equation, we have
    \begin{align*}
        0=&\int_{\Omega}E^{-1}(\chi)\sigma:\hat{\sigma}\ dx+\int_{\Omega}(\nabla\cdot\hat{\sigma})\cdot\lambda\ dx+\int_{\Gamma_N}\hat{\sigma}\cdot\vec{n}\cdot\mu\ ds\\
        =&\int_{\Omega}E^{-1}(\chi)\sigma:\hat{\sigma}\ dx+\int_{\Gamma_N\cup\Gamma_D}\hat{\sigma}\cdot\vec{n}\cdot\lambda\ dx-\int_{\Omega}\hat{\sigma}:\nabla\lambda\ dx+\int_{\Gamma_N}\hat{\sigma}\cdot\vec{n}\cdot\mu\ ds\\
        =&\int_{\Omega}(E^{-1}(\chi)\sigma-\varepsilon(\lambda)):\hat{\sigma}\ dx+\int_{\Gamma_D}\hat{\sigma}\cdot\vec{n}\cdot\lambda\ dx+\int_{\Gamma_N}\hat{\sigma}\cdot\vec{n}\cdot(\lambda+\mu).
    \end{align*}
Choosing arbitrary $\hat{\sigma}\in S_0$ we can derive that
     $E^{-1}(\chi)\sigma=\varepsilon(\lambda).$
    Setting $\hat{\sigma}\in S_0$ and $\hat{\sigma}\cdot\vec{n}|_{\Gamma_N}=0$ implies
    $\lambda =0$ on $\Gamma_D.$
    Lastly, setting $\hat{\sigma}\in H(\mbox{div};\Omega)$ yields
    $ \lambda = -\mu$ on $\Gamma_N.$
     Thus, the result is proven, where $ \lambda$ and $\mu $ are Lagrange multipliers.
\end{proof}
Here, we should note that the problem \eqref{stress_problem} is a strictly convex optimization problem, so the solution is unique (if a solution exists), and the first-order optimality condition is necessary and sufficient. On the other hand, the first-order optimality condition corresponds to the linear elasticity equation. Thus, we establish the existence and uniqueness of the solution to the problem \eqref{stress_problem}. This also demonstrates the existence and uniqueness of the solution to the optimization problem \(\min_{(\hat{\sigma}, \hat{\varrho})\in S} \tilde{g}(\chi,\hat{\sigma},\hat{\varrho})\). Hence, we can prove the following theorem.
\begin{theorem}
    The problem \eqref{orign_problem} is equivalent to the problem \eqref{bilevel_modify}, as they are different formulations of the same problem. Specifically, \eqref{orign_problem} is formulated with the displacement field as the state variable, while \eqref{bilevel_modify} is formulated with the stress field as the state variable.
\end{theorem}
\begin{proof}
    Let $u^*$ and $v^*$ be solutions of the state equation \eqref{elasticity_state} and the adjoint state equation \eqref{elasticity_adjoint}, respectively. We have established in Lemma \ref{stress_equ} that the constraint of problem \eqref{bilevel_modify} is essentially equivalent to solving two linear elasticity equations, where \(\sigma^*=E(\chi)\varepsilon(u^*)\) and \(\varrho^*=E(\chi)\varepsilon(v^*)\) are the stress fields corresponding to the state equation \eqref{elasticity_state} and the adjoint equation \eqref{elasticity_adjoint}. Therefore, we have
    \begin{align*}
        \tilde{J}(\chi,\sigma^*,\varrho^*)=\int_\Omega E^{-1}(\chi)\sigma^*:\varrho^*\ dx
        =\int_\Omega E(\chi)\varepsilon(u^*):\varepsilon(v^*)\ dx
        =a(\chi;u^*,v^*)
        =\ell_{\rm out}(u^*)
        =J(\chi,u^*),
    \end{align*}
    which demonstrates the equivalence between problem \eqref{orign_problem} and problem \eqref{bilevel_modify}.
\end{proof}
In the following, we prove the equivalence between the problem \eqref{bilevel_modify} and the penalty form \eqref{penalty_formule}.
\begin{theorem}
    Provided $ \lambda > 1/2 $, it holds that 
    \begin{equation} 
        \min_ {(\hat\sigma,\hat\varrho)\in S}L(\chi,\hat\sigma,\hat\varrho) = \tilde{J}(\chi,\sigma^*,\varrho^*),
    \end{equation}
    where $(\sigma^*,\varrho^*)\in \arg\min\limits_{(\hat{\sigma},\hat{\varrho})\in S}\tilde{g}(\chi,\hat\sigma,\hat\varrho)$.
\end{theorem}
\begin{proof}
    First, let $u^*$ and $v^*$ be the solutions of \eqref{elasticity_state} and \eqref{elasticity_adjoint}, respectively. We observe that
    \begin{equation}
    \sigma^*=E(\chi)\varepsilon(u^*),\quad \varrho^*=E(\chi)\varepsilon(v^*).
    \end{equation}
    Next, we prove that when $ \lambda > 1/2 $, $L(\chi,\sigma^*,\varrho^*)= \min\limits_{(\hat\sigma,\hat\varrho)\in S}L(\chi,\hat\sigma,\hat\varrho)$.  
    Note that
    \begin{align*}
        L(\chi,\hat\sigma,\hat\varrho)-L(\chi,\sigma^*,\varrho^*)
        &=\int_{\Omega}E^{-1}(\chi)\hat\sigma:\hat\varrho\ dx +\lambda\Big(\int_{\Omega}E^{-1}(\chi)(\hat\sigma:\hat\sigma+\hat\varrho:\hat\varrho)\ dx\Big)\\
        &-\lambda\Big(\int_{\Omega}E^{-1}(\chi)(\sigma^*:\sigma^*+\varrho^*:\varrho^*)\ dx\Big)-\int_{\Omega}E^{-1}(\chi)\sigma^*:\varrho^*\ dx\\
        &=\dfrac{1}{2}\int_{\Omega}E^{-1}(\chi)(\hat\sigma+\hat\varrho):(\hat\sigma+\hat\varrho)\ dx+\Big(\lambda-\dfrac{1}{2}\Big)\tilde{g}(\chi,\hat\sigma,\hat\varrho)\\
        &-\dfrac{1}{2}\int_{\Omega}E^{-1}(\chi)(\sigma^*+\varrho^*):(\sigma^*+\varrho^*)\ dx-\Big(\lambda-\dfrac{1}{2}\Big)\tilde{g}(\chi,\sigma^*,\varrho^*).
    \end{align*}
    Since $(\sigma^*,\varrho^*)\in \arg\min_{(\hat{\sigma},\hat{\varrho})\in S} \tilde{g}(\chi,\hat{\sigma},\hat{\varrho})$, we only need to prove that
    $$(\sigma^*+\varrho^*)\in\arg\min_{(\hat\sigma,\hat\varrho)\in S}\int_{\Omega}E^{-1}(\chi)(\hat\sigma+\hat\varrho):(\hat\sigma+\hat\varrho)\ dx.$$
Indeed, it follows from the fact $(\sigma^*,\varrho^*),(\hat\sigma,\hat\varrho)\in S$ that 
    \begin{align*}
        &\int_{\Omega}E^{-1}(\chi)(\hat\sigma+\hat\varrho):(\hat\sigma+\hat\varrho)\ dx-\int_{\Omega}E^{-1}(\chi)(\sigma^*+\varrho^*):(\sigma^*+\varrho^*)\ dx\\
        =&\int_{\Omega}E^{-1}(\chi)|\hat\sigma+\hat\varrho-\sigma^*-\varrho^*|^2\ dx + 2\int_{\Omega}E^{-1}(\chi)(\sigma^*+\varrho^*):(\hat\sigma+\hat\varrho-\sigma^*-\varrho^*)\ dx\\
        =&\int_{\Omega}E^{-1}(\chi)|\hat\sigma+\hat\varrho-\sigma^*-\varrho^*|^2\ dx + 2\int_{\Omega}(\varepsilon(u^*)+\varepsilon(v^*)):(\hat\sigma+\hat\varrho-\sigma^*-\varrho^*)\ dx\\
        =&\int_{\Omega}E^{-1}(\chi)|\hat\sigma+\hat\varrho-\sigma^*-\varrho^*|^2\ dx - 2\int_{\Omega}(u^*+v^*)\cdot\nabla\cdot(\hat\sigma+\hat\varrho-\sigma^*-\varrho^*)\ dx\\
        &+2\int_{\Gamma_D\cup\Gamma_N}(u^*+v^*)\cdot(\hat\sigma+\hat\varrho-\sigma^*-\varrho^*)\cdot\vec{n}\ ds\\
        =&\int_{\Omega}E^{-1}(\chi)|\hat\sigma+\hat\varrho-\sigma^*-\varrho^*|^2\ dx\geq 0.
    \end{align*}
    The inequality above indicates that we have identified a minimizer so that the problem has a solution. Therefore, further discussion of the existence of solutions is unnecessary. This finishes the proof.
\end{proof}
Therefore, we can immediately show that the compliant mechanism problem \eqref{orign_problem} is equivalent to the following double minimization problem:
\begin{equation}
    \min_{\chi\in\mathcal{U}}\min_{(\sigma,\varrho)\in S}L(\chi,\sigma,\varrho).
\end{equation}
Moreover, the solution of the inner minimization problem can be obtained by solving the state equation
\begin{equation}
    a(\chi;u,\hat{u})=\ell_{\rm in}(\hat{u})\quad\forall\hat{u}\in H_D^1(\Omega)^d,
\end{equation}
the adjoint state equation
\begin{equation}
    a(\chi;v,\hat{v})=\ell_{\rm out}(\hat{v})\quad\forall\hat{v}\in H_D^1(\Omega)^d,
\end{equation}
and set \(\sigma = E(\chi)\varepsilon(u)\), \(\rho = E(\chi)\varepsilon(v)\).

\subsection{Derivation of the optimization algorithm}
Since $\mathcal{U}$ is not a convex set, it is challenging to analyze the properties of the objective functional with respect to \(\chi\). To do this, we first relax $ \chi $ to its convex hull $\bar{\mathcal{U}}$ and consider the following double minimization problem
\begin{equation}
    \min\limits_{\chi\in \bar{\mathcal{U}}}\min\limits_{(\sigma,\rho)\in S}L(\chi,\sigma,\varrho):=\int_{\Omega}E^{-1}(\chi)\sigma:\varrho\ dx + \lambda\left( \tilde{g}(\chi,\sigma,\varrho)-\min_{(\hat{\sigma},\hat{\varrho})\in S} \tilde{g}(\chi,\hat{\sigma},\hat{\varrho})\right),
\end{equation}
where $$ \bar{\mathcal{U}}:=\Big\{\chi\in BV(\Omega):\ \chi\in [0,1],\int_{\Omega}\chi\ dx \leq\beta|\Omega|\Big\}.$$ 

It should be noted that, due to the relaxation of $ \chi$ from $\mathcal{U}$ to $ \bar{\mathcal{U}} $, some previously valid equations no longer hold, such as:
\begin{align}
    a(\chi;u,v)=\int_{\Omega}E(\chi)\varepsilon(u):\varepsilon(v)\ dx\neq\int_{\Omega}E^{-1}(\chi)\sigma:\varrho\ dx\quad\mbox{for\ some}\  \chi\in\bar{\mathcal{U}},
\end{align}
because $E^{-1}(\chi)\neq E(\chi)^{-1}$ when $\chi\in [0,1]$. Therefore, for $ \chi $ belonging to the closed convex set $\bar{\mathcal{U}}$, we introduce the following notation: 
\begin{equation}
    E(\chi)=\dfrac{1}{A(\chi)}{E}_0,\quad E^{-1}(\chi)=A(\chi){E}_0^{-1},
\end{equation}
 and
$$ A(\chi)=\Big(\dfrac{1}{E_{\max}}-\dfrac{1}{E_{\min}}\Big)\chi+\dfrac{1}{E_{\min}} .$$
Clearly, when $ \chi \in \mathcal{U} $ is the indicator function, the following equality is maintained
$$\dfrac{1}{A(\chi)} = (E_{\max}-E_{\min})\chi + E_{\min}.$$
Therefore, these definitions are reasonable extensions of $ E^{-1}(\chi) $ and $ E(\chi) $ defined in Section 2.1.

Throughout the remainder of this paper, we redefine \( a(\chi; u, v): L^\infty(\Omega) \times H_D^1(\Omega)^d \times H_D^1(\Omega)^d \to \mathbb{R} \) as follows
\begin{equation}
    a(\chi; u, v):=\int_\Omega\dfrac{1}{A(\chi)}{E}_0\varepsilon(u):\varepsilon(v)\ dx.
\end{equation}
Of course, such an extension also introduces a problem, namely that the optimization problem formulated in terms of the stress field and the optimization problem formulated in terms of the displacement field are essentially two distinct problems, with different properties on \(\mathcal{U}\). This is precisely the reason why transforming the problem to the stress field formulation yields better numerical performance. However, in the section on theoretical analysis, we need to clearly articulate the relationship between each step of the transformation.

First, we derive the Fréchet derivative of $ L $ with respect to $ \chi $.
\begin{lemma}\label{F-derivative}
    Define $$G(\chi) = \min_{(\hat\sigma,\hat\varrho)\in S} \tilde{g}(\chi,\hat\sigma,\hat\varrho):=\int_{\Omega}A(\chi)E_0^{-1}(\hat\sigma:\hat\sigma+\hat\varrho:\hat\varrho)\ dx,$$
    then from Lemma \ref{stress_equ} we have
    \begin{eqnarray}
        G(\chi) =\int_\Omega A(\chi)E_0^{-1}\big(\sigma(\chi):\sigma(\chi)+\varrho(\chi):\varrho(\chi)\big)\ dx.\nonumber
    \end{eqnarray}
    Here we denote the stress fields corresponding to \(\chi\) as \(\sigma(\chi):=\dfrac{1}{A(\chi)}{E}_0\varepsilon(u)\) and \(\varrho(\chi):=\dfrac{1}{A(\chi)}{E}_0\varepsilon(v)\) to emphasize their dependence on \(\chi\),
%    \begin{eqnarray}
%    (\sigma(\chi),\varrho(\chi)):=(\dfrac{1}{A(\chi)}{E}_0\varepsilon(u),\dfrac{1}{A(\chi)}{E}_0\varepsilon(v)),\nonumber    
%    \end{eqnarray}
    where $ u\in H_D^1(\Omega)^d $ and $ v\in H_D^1(\Omega)^d $ satisfy
    \begin{eqnarray} \label{u_weak}
    \int_{\Omega}\dfrac{1}{A(\chi)}{E}_0\varepsilon(u):\varepsilon(\hat{u})\ dx=\int_{\Gamma_N}q_{\rm in}\cdot\hat{u}\ ds\quad\forall\hat{u}\in H_D^1(\Omega)^d
    \end{eqnarray}
    and 
    \begin{eqnarray} \label{v_weak}
    \int_{\Omega}\dfrac{1}{A(\chi)}{E}_0\varepsilon(v):\varepsilon(\hat{v})\ dx=\int_{\Gamma_N}q_{\rm out}\cdot\hat{v}\ ds\quad\forall\hat{v}\in H_D^1(\Omega)^d.
    \end{eqnarray}
    Moreover, there holds
    $$\dfrac{\partial G}{\partial \chi}(\chi)(\hat{\chi})=\int_{\Omega}A'(\chi)(\hat{\chi}){E}_0^{-1}(\sigma(\chi):\sigma(\chi)+\varrho(\chi):\varrho(\chi))\ dx.$$
\end{lemma}
\begin{proof}
    From the definition of $\sigma(\chi)$ and $\rho(\chi)$, $G(\chi)$ can be transformed into
    \begin{eqnarray}
    G(\chi)=\int_{\Omega}\dfrac{1}{A(\chi)}{E}_0(\varepsilon(u):\varepsilon(u)+\varepsilon(v):\varepsilon(v))\ dx.\label{G_rep}    
    \end{eqnarray}
    Since the two components of $ G(\chi) $ have identical forms, we simply consider the following functional
    $$j(\chi) = \int_{\Omega}\dfrac{1}{A(\chi)}{E}_0\varepsilon(u):\varepsilon(u)\ dx$$
    subject to
    $$ \int_{\Omega}\dfrac{1}{A(\chi)}{E}_0\varepsilon(u):\varepsilon(\hat{u})\ dx=\int_{\Gamma_N}q_{\rm in}\cdot\hat{u}\ ds\quad \forall \hat{u}\in H_D^1(\Omega)^d.$$
    Hence, its Lagrange functional reads as follows
    $$\ell(\chi,u,\hat{u})=\int_{\Omega}\dfrac{1}{A(\chi)}{E}_0\varepsilon(u):\varepsilon(u)\ dx+\int_{\Omega}\dfrac{1}{A(\chi)}{E}_0\varepsilon(u):\varepsilon(\hat{u})\ dx-\int_{\Gamma_N}q_{\rm in}\cdot\hat{u}\ ds.$$
    It is straightforward to show that $\hat u=-2u$. Therefore, 
    $$\dfrac{\partial j}{\partial \chi}(\chi)(\hat{\chi})=\int_{\Omega}\dfrac{A'(\chi)\hat{\chi}}{A(\chi)^2}E_0\varepsilon(u):\varepsilon(u)\ dx.$$
    Furthermore, since $\sigma(\chi)=\dfrac{1}{A(\chi)}{E}_0\varepsilon(u)$, it follows that
    $$\dfrac{\partial j}{\partial \chi}(\chi)(\hat{\chi})=\int_{\Omega}A'(\chi)(\hat{\chi}){E}_0^{-1}\sigma(\chi):\sigma(\chi)\ dx.$$
    Hence,
    \begin{eqnarray}
    \dfrac{\partial G}{\partial \chi}(\chi)(\hat{\chi})&=&\int_{\Omega}\dfrac{A'(\chi)(\hat{\chi})}{A(\chi)^2}{E}_0(\varepsilon(u):\varepsilon(u)+\varepsilon(v):\varepsilon(v))\ dx\nonumber\\
    &=&\int_{\Omega}A'(\chi)(\hat{\chi}){E}_0^{-1}(\sigma(\chi):\sigma(\chi)+\varrho(\chi):\varrho(\chi))\ dx.\nonumber
    \end{eqnarray} 
    This finishes the proof.
\end{proof}

From Lemma \ref{F-derivative}, the Fréchet derivative of $ L $ is obtained directly as follows:
\begin{align}
    \dfrac{\partial L}{\partial \chi}(\chi,\sigma,\varrho)(\hat{\chi}) =&\int_{\Omega}A'(\chi)(\hat{\chi}){E}_0^{-1}\sigma:\varrho\ dx +
    \lambda\Big(\int_{\Omega}A'(\chi)(\hat{\chi}){E}_0^{-1}(\sigma:\sigma+\varrho:\varrho)\ dx\notag\\
    &-\int_{\Omega}A'(\chi)(\hat{\chi}){E}_0^{-1}(\sigma(\chi):\sigma(\chi)+\varrho(\chi):\varrho(\chi))\ dx\Big).
\end{align}

In particular, if $(\sigma,\varrho)=(\sigma(\chi),\varrho(\chi))$, then
\begin{equation}\label{obj_derivative}
    \dfrac{\partial L}{\partial \chi}(\chi,\sigma(\chi),\varrho(\chi))(\hat{\chi}) =\int_{\Omega}A'(\chi)(\hat{\chi}){E}_0^{-1}\sigma(\chi):\varrho(\chi)\ dx.
\end{equation}
    
Next, we demonstrate that $ \min_{(\sigma,\varrho)\in S}\tilde{g}(\chi,\sigma,\varrho) $ is concave with respect to $ \chi $. We only need to prove the following lemma.
\begin{lemma}\label{F-second-derivative}
    The functional 
    $$g_1(\chi)=\min_{\sigma\in S_1}\int_{\Omega}E^{-1}(\chi)\sigma:\sigma\ dx$$ 
    is concave with respect to $ \chi $.
\end{lemma}
\begin{proof}
     It follows from the proof of Lemma \ref{F-derivative} that 
    \begin{align*}
g_1(\chi)=\int_{\Omega}\dfrac{1}{A(\chi)}{E}_0\varepsilon(u):\varepsilon(u)\ dx,
    \end{align*}
    where $u$ solves the state equation \eqref{u_weak}. Here, we do not employ the Lagrangian approach, but instead use the implicit function theorem. Denoting \( u' = \dfrac{\partial u}{\partial\chi}(\hat{\chi}) \) and \( u'' = \dfrac{\partial^2 u}{\partial\chi^2}(\hat{\chi}^2) \), by the chain rule, we can compute
    \begin{align*}
        \dfrac{\partial^2 g_1}{\partial\chi^{2}}(\hat{\chi}^2)=& 2\int_{\Omega}\dfrac{[A'(\chi)(\hat{\chi})]^2}{A(\chi)^{3}}{E}_0\varepsilon(u):\varepsilon(u)\ dx-4\int_{\Omega}\dfrac{A'(\chi)(\hat{\chi})}{A(\chi)^2}{E}_0\varepsilon(u):\varepsilon(u')\ dx\\
        &+2\int_{\Omega}\dfrac{1}{A(\chi)}{E}_0\varepsilon(u'):\varepsilon(u')\ dx+2\int_{\Omega}\dfrac{1}{A(\chi)}{E}_0\varepsilon(u):\varepsilon(u'')\ dx.
    \end{align*}
    Differentiating with respect to $\chi$ in \eqref{u_weak}, we obtain 
    \begin{equation}
        -\int_{\Omega}\dfrac{A'(\chi)(\hat{\chi})}{A(\chi)^{2}}{E}_0\varepsilon(u):\varepsilon(\hat{u})\ dx+\int_{\Omega}\dfrac{1}{A(\chi)}{E}_0\varepsilon(u'):\varepsilon(\hat{u})\ dx=0,\nonumber
    \end{equation}
    \begin{align*}
        2\int_{\Omega}\dfrac{[A'(\chi)(\hat{\chi})]^2}{A(\chi)^{3}}{E}_0\varepsilon(u):\varepsilon(\hat{u})\ dx-2\int_{\Omega}\dfrac{A'(\chi)(\hat{\chi})}{A(\chi)^2}{E}_0\varepsilon(u'):\varepsilon(\hat{u})\ dx
        +\int_{\Omega}\dfrac{1}{A(\chi)}{E}_0\varepsilon(u''):\varepsilon(\hat{u})\ dx=0
    \end{align*}
   for any $\hat{u}\in H_D^1(\Omega)^d$.
    Hence, setting $\hat{u} = u'$ in the first equation yields
    \begin{equation}\label{state_tangent}
        \int_{\Omega}\dfrac{1}{A(\chi)}{E}_0\varepsilon(u'):\varepsilon(u')\ dx=\int_{\Omega}\dfrac{A'(\chi)(\hat{\chi})}{A(\chi)^{2}}{E}_0\varepsilon(u):\varepsilon(u')\ dx,
    \end{equation} 
    while setting $\hat{u} = u$ in the second equation gives
    \begin{align*}
        \int_{\Omega}\dfrac{1}{A(\chi)}{E}_0\varepsilon(u''):\varepsilon(\hat{u})\ dx
        =-2\int_{\Omega}\dfrac{[A'(\chi)(\hat{\chi})]^2}{A(\chi)^{3}}{E}_0\varepsilon(u):\varepsilon(u)\ dx+2\int_{\Omega}\dfrac{A'(\chi)(\hat{\chi})}{A(\chi)^2}{E}_0\varepsilon(u'):\varepsilon(u)\ dx.
    \end{align*}
    By applying the Cauchy-Schwarz inequality in \eqref{state_tangent}, we obtain
    \begin{equation}
        \int_{\Omega}\dfrac{1}{A(\chi)}{E}_0\varepsilon(u'):\varepsilon(u')\ dx\leq \int_{\Omega}\dfrac{[A'(\chi)(\hat{\chi})]^2}{A(\chi)^{3}}{E}_0\varepsilon(u):\varepsilon(u)\ dx.
    \end{equation}
    Therefore
    \begin{align*}
      \dfrac{\partial^2 g_1(\chi,u,v)}{\partial \chi^2}(\hat{\chi})(\hat{\chi})=-2\int_{\Omega}\dfrac{[A'(\chi)(\hat{\chi})]^2}{A(\chi)^{3}}{E}_0\varepsilon(u):\varepsilon(u)\ dx+2\int_{\Omega}\dfrac{1}{A(\chi)}{E}_0\varepsilon(u'):\varepsilon(u')\ dx\leq 0.
    \end{align*}
This finishes the proof.
\end{proof}

Since $ \tilde{J}$ and $\tilde{g} $ are linear functionals with respect to $\chi$, it follows from Lemma \ref{F-second-derivative} that $ L(\chi,\sigma,\rho)=\tilde{J}(\chi,\sigma,\rho)+\lambda(\tilde{g}(\chi,\sigma,\rho)-\min_{(\hat{\sigma},\hat{\varrho})\in S}\tilde{g}(\chi,\hat{\sigma},\hat{\rho})) $ is convex with respect to $ \chi $.

To ensure the existence of solutions, we need to add a perimeter penalty term to \( L \). We consider the regularized problem 
\begin{equation}\label{L_permeter}
    L^{\gamma}(\chi,\sigma,\varrho)=L(\chi,\sigma,\varrho)+\gamma\|\chi\|_{TV},
\end{equation}
where \(\|\cdot\|_{TV}\) represents the total variation norm in the space of functions with bounded variation. We remark that the TV-norm  is computationally challenging, so  one usually uses an approximate perimeter regularization for the implementation. To maintain the convexity of the objective functional, we introduce a perimeter approximation, resulting in the following approximate objective functional:
\begin{equation}\label{L_permeter_approx}
    L^{\gamma,\epsilon}(\chi,\sigma,\varrho)=L(\chi^\epsilon,\sigma,\varrho)+\dfrac{\gamma}{\epsilon}P_{G_\epsilon}(\chi),
\end{equation}
where \(\chi^\epsilon := G_\epsilon * \chi\) and the family \(\{G_\epsilon\in C_c^\infty(\mathbb{R}^d)\}_\epsilon\) is a Dirac sequence.

Specifically, we assume that \( G \) is a smooth kernel satisfying \( \int_{\mathbb{R}^d} G(x) \, dx = 1 \). We define \( G_\epsilon(x) = \dfrac{1}{\epsilon^d} G\left( \dfrac{x}{\epsilon} \right) \), which is introduced for the need of theoretical analysis, and
\begin{equation}
     P_{G_\epsilon}(\chi):=\dfrac{1}{2}\int_{\Omega}\int_{\Omega}G_\epsilon(x-y)|\chi(x)-\chi(y)|\ dxdy.
\end{equation}
 It can be shown that the nonlocal convex approximation $\dfrac{C}{\epsilon}P_{G_\epsilon}(\chi)$, with \( C>0 \) a constant, $\Gamma$-converges to $\|\chi\|_{TV}$, we refer to \cite{Maz} for more details. We can provide a subderivative of this functional as follows:
\begin{equation}
    \partial P_{G_\epsilon}(\chi) = \Big\{ g \in L^\infty(\Omega) \mid \langle g,\tilde{\chi}\rangle =  \int_\Omega \Big(\int_\Omega G_\epsilon(x-y) \xi(x,y) dy\Big)\tilde{\chi}(x)\ dx, \ \xi(x,y) \in \text{sign}(\chi(x) - \chi(y)) \Big\}.
\end{equation}
 Consequently, we obtain an optimization problem \( P \) that can be executed efficiently
 \begin{equation}\label{problem_P}
    \min_{(\chi,\sigma,\rho)\in \mathcal{U}\times S}L^{\gamma,\epsilon}(\chi,\sigma,\varrho)=\int_\Omega A(\chi^\epsilon)E^{-1}_0\sigma:\varrho\ dx + \lambda\Big(\int_\Omega A(\chi^\epsilon)E^{-1}_0(|\sigma|^2+|\varrho|^2-|\sigma^*|^2-|\varrho^*|^2)\ dx\Big) +\dfrac{\gamma}{\epsilon} P_{G_\epsilon}(\chi).
\end{equation}

 We emphasize again that the inner minimizer of this double minimization problem can be obtained by solving the corresponding linear elasticity equation, while the outer minimization is a nonconvex optimization problem. However, we have established that the objective functional $L^{\gamma,\epsilon}$ is convex on \(\bar{\mathcal{U}}\). Therefore, we adopt the projected gradient algorithm. 
 
 Let \(\chi_k\) be the current iteration point and \((\sigma_k, \varrho_k)\) be the solution to the corresponding inner minimization problem. We first perform a gradient descent update in \(L^\infty(\Omega)\): 
 \begin{equation}\label{proximal_update}
     \bar\chi_{k+1} = \arg\min_{\chi\in L^\infty(\Omega)}\Big\{L^{\gamma,\epsilon}(\chi_k,\sigma_k,\varrho_k)+\Big\langle\dfrac{\partial L^{\gamma,\epsilon}}{\partial \chi}(\chi_k,\sigma_k,\varrho_k),\chi-\chi_k\Big\rangle +\dfrac{1}{2r}\|\chi-\chi_k\|_{0,\Omega}^2\Big\}.
 \end{equation}
 Then, we project the resulting minimizer \(\bar{\chi}_{k+1}\) onto \(\mathcal{U}\)
 \begin{equation}\label{project_update}
     \chi_{k+1}=\mathcal{P}_\mathcal{U}(\bar\chi_{k+1}).
 \end{equation}
In the following, we derive the explicit computational forms for the update of the gradient descent \eqref{proximal_update} and the projection \eqref{project_update}. First, for the update of the gradient descent, the key lies in the selection of the subgradient. Here, we choose
\begin{equation}
\langle g_k,\tilde{\chi}\rangle =  \int_\Omega \Big(\int_\Omega G_\epsilon(x-y) (\chi_k(x)-\chi_k(y))\ dy\Big)\tilde{\chi}(x)\ dx,\nonumber
\end{equation}
so that $g_k\in \partial P_{G_\epsilon}(\chi)$. Recalling \eqref{obj_derivative},  the calculation proceeds as follows:
\begin{align*}
    &\Big\langle \dfrac{\partial L^{\gamma,\epsilon}}{\partial \chi}(\chi_k,\sigma_k,\varrho_k),\chi-\chi_k\Big\rangle\\
    &=\Big\langle \dfrac{\partial L(\chi_k^\epsilon,\sigma_k,\varrho_k)}{\partial\chi},\chi-\chi_k\Big\rangle + \dfrac{\gamma}{\epsilon}\langle g_k,\chi-\chi_k\rangle\\
    &=\int_\Omega A'(G_\epsilon*(\chi-\chi_k))E_0^{-1}\sigma_k:\varrho_k\ dx + \dfrac{\gamma}{\epsilon}\int_\Omega \Big(\int_\Omega G_\epsilon(x-y) (\chi_k(x)-\chi_k(y)) dy\Big)(\chi(x)-\chi_k(x))\ dx\\
    &=\int_\Omega(\chi-\chi_k)\Big(\dfrac{1}{E_{\max}}-\dfrac{1}{E_{\min}}\Big)G_\epsilon*(E_0^{-1}\sigma_k:\varrho_k)\ dx +\dfrac{\gamma}{\epsilon}\int_\Omega (\chi-\chi_k)(\chi_k-G_\epsilon*\chi_k)\ dx.
\end{align*}
The last equality is obtained by extending the function to zero and then using the properties of convolution. Hence, we have
\begin{equation}\label{proxiaml_update_exp}
    \bar\chi_{k+1}=\chi_k - r\Big( \Big(\dfrac{1}{E_{\max}}-\dfrac{1}{E_{\min}}\Big)G_\epsilon*(E_0^{-1}\sigma_k:\varrho_k) +\dfrac{\gamma}{\epsilon}(\chi_k-G_\epsilon*\chi_k)\Big).
\end{equation}
Next, we address the computational issue of the projection. The key aspect of the projection lies in the choice of norm under which the projection is defined, as different norms lead to different projection results, directly impacting our study of convergence. Here, we use the \( L^1 \) norm.

It is easy to see that for any \( L^p \) norm, the projection result is an element defined as follows: Let \( c_{k+1} = \inf \{ c \in \mathbb{R} : \mu(\{ x : \bar{\chi}_{k+1}(x) > c \}) \leq \beta |\Omega| \} \), then
\begin{equation}\label{project_update_exp}
    \chi_{k+1}(x) =
    \begin{cases} 
    1 & \text{if } \bar{\chi}_{k+1}(x) > c_{k+1}, \\
    0 & \text{if } \bar{\chi}_{k+1}(x) \leq c_{k+1},
    \end{cases}
\end{equation}
where \(\mu\) is a measure of \(\mathbb{R}^d\).

With the update formulas \eqref{proxiaml_update_exp} and \eqref{project_update_exp}, we can easily obtain an estimate \(\chi_{k+1}\). However, the step size \( r \) should be carefully chosen. Therefore, we need to adjust the size of \( r \) to obtain a suitable \(\chi_{k+1}\). Here, we do not adopt the classical Wolfe conditions, as the outer optimization problem at the discrete level is a 0-1 programming problem. It is sufficient to ensure that the objective function decreases. Thus, we use a binary search to find an appropriate \( r \), thereby completing an iteration step.

Clearly, the choice of the search interval for \( r \) significantly affects the number of iterations, which in turn affects the efficiency of the algorithm. By carefully examining the update and projection steps of the projected gradient algorithm, we observe that if the update does not alter the relative order of the points where \(\chi_k\) takes values 1 and 0, then \(\bar{\chi}_{k+1} = \chi_k\). Consequently, we can choose
$$ r_{\min} = \dfrac{1}{2} \sup_{x\in\Omega}\Big( \Big(\dfrac{1}{E_{\max}}-\dfrac{1}{E_{\min}}\Big)G_\epsilon*(E_0^{-1}\sigma_k:\varrho_k) +\dfrac{\gamma}{\epsilon}(\chi_k-G_\epsilon*\chi_k)\Big)^{-1},$$
while the choice of \( r_{\max} \) is relatively arbitrary, as long as \( r_{\max} > r_{\min} \). The \(\bar{\chi}_{k+1}\) obtained from \( r_{\max} \) falls into one of the following three cases: 
\begin{enumerate}
    \item If \(\bar{\chi}_{k+1}\) reduces the objective function, the search can be terminated and we set \(\chi_{k+1} \leftarrow \bar{\chi}_{k+1}\); 
    \item If \(\bar{\chi}_{k+1} = \chi_k\), we update \( r_{\min} \leftarrow r_{\max} \) and select a larger \( r_{\max} \); 
    \item If \(\bar{\chi}_{k+1}\) increases the objective function, it indicates that \( r_{\max} \) is too large, and we can perform a line search within this interval.
\end{enumerate}

\subsection{Numerical implementation of the algorithm}
In the following, we present the numerical implementation of the algorithm in pseudo-code form.

\begin{algorithm}
\caption{The step-size line search algorithm}
\label{alg:coefficient_line_search}
\begin{algorithmic}[1]
\STATE Initialize $r_{\min},r_{\max},\Delta$, and let $r = \dfrac{r_{\max}+r_{\min}}{2}$.
\WHILE{1}
    \STATE Compute $\bar\chi$ using \eqref{proxiaml_update_exp}.
    \STATE Compute $\bar{\chi}_{k+1}$ using \eqref{project_update_exp}.
    \STATE Compute $L^{\gamma,\epsilon}(\bar{\chi}_{k+1},\sigma_k,\varrho_k)$
    \STATE \IF{$L^{\gamma,\epsilon}(\bar{\chi}_{k+1},\sigma_k,\varrho_k)> L^{\gamma,\epsilon}(\chi_k,\sigma_k,\varrho_k)$}
    \STATE  $r_{\max} \leftarrow r$.
    \STATE\IF{$\|\bar{\chi}_{k+1}-\chi_k\|_1 \leq \Delta$}
    \STATE $\chi_{k+1} \leftarrow \chi_k$, break;
    \STATE\ENDIF
    \STATE \ELSIF{$L^{\gamma,\epsilon}(\bar{\chi}_{k+1},\sigma_k,\varrho_k)< L^{\gamma,\epsilon}(\chi_k,\sigma_k,\varrho_k)$}
    \STATE $\chi_{k+1} \leftarrow \bar{\chi}_{k+1}$, break;
    \STATE \ELSIF{$L^{\gamma,\epsilon}(\bar{\chi}_{k+1},\sigma_k,\varrho_k)= L^{\gamma,\epsilon}(\chi_k,\sigma_k,\varrho_k)$ and $\|\bar{\chi}_{k+1}-\chi_k\|_1 > \Delta$}
    \STATE $r_{\max} \leftarrow r$.
    \STATE \ENDIF
    \STATE\IF{$\bar{\chi}_{k+1} == \chi^k$} 
    \STATE $r_{\min} \leftarrow r$.
    \STATE \ENDIF
    \STATE $r=\dfrac{r_{\max}+r_{\min}}{2}$.
\ENDWHILE
\RETURN New density \( \chi_{k+1} \) and function value $L^{\gamma,\epsilon}(\chi_{k+1},\sigma_k,\varrho_k)$.
\end{algorithmic}
\end{algorithm}

It should be noted that we introduce a parameter \(\Delta\) here, which serves to terminate the line search and return \(\chi_k\) when a change of \(\chi_k\) by \(\Delta\) (in the sense of \(L^1\)) along the descent direction leads to an increase in the value of the objective function. This effectively terminates the entire optimization algorithm. At the discrete level, we can choose \(\Delta = C \mu(h)\), where \(h\) represents the mesh size of finite elements.

\begin{algorithm}
\caption{Gradient desecent algorithm for the compliant mechanisms problem}
\label{alg:monotonic_descen_with_penlty}
\begin{algorithmic}[1]
\STATE Initialize $k=0,\chi_k$ (The initial value is typically set to a uniform distribution with $ \chi_k = \beta $),  and parameters $\lambda,\gamma,\epsilon,\Delta$.
\WHILE{\( \chi_k \) not converged}
    \STATE Solve  $a(\chi_k^\epsilon;u,\hat{u})=\ell_{\rm in}(\hat{u}),\quad\forall\hat{u}\in H_D^1(\Omega)^d$ and $a(\chi_k^\epsilon;v,\hat{v})=\ell_{\rm out}(\hat{u}),\quad\forall\hat{v}\in H_D^1(\Omega)^d$.
    \STATE Compute $\Big(\dfrac{1}{E_{\max}}-\dfrac{1}{E_{\min}}\Big)G_\epsilon*(E_0^{-1}\sigma_k:\varrho_k) +\dfrac{\gamma}{\epsilon}(\chi_k-G_\epsilon*\chi_k)$.
    \STATE Update \( \chi_{k}\) using Algorithm \ref{alg:coefficient_line_search}.
    \STATE $k\leftarrow k+1$.
\ENDWHILE
\RETURN Optimal density \( \chi^* \) and function value $L^{\gamma,\epsilon}(\chi^*,\sigma^*,\varrho^*)$.
\end{algorithmic}
\end{algorithm}

%\newpage
\section{Theoretical Analysis of Convergence}\label{sec:theory}
 In this section, we investigate the convergence of the proposed algorithm from an abstract perspective. First, we standardize the notation: variables with subscript \( n \) represent those in the discrete problem, which vary with the grid size \( h_n \), while variables without subscript \( n \) denote those in the continuous problem. To clearly present our analysis process and for the convenience of notation, we restate the three optimization problems.  We introduce the notation \( X_n = \mathcal{U}_n \times S_n\) and \( X = \mathcal{U} \times S \), where \( \mathcal{U}_n \subset \mathcal{U}\) consists of piecewise constant finite element functions and \( S_n \) denotes the finite element space that is conforming. 
 
\subsection{Problem description}
The origin problem $Q$ is defined as follows:
\begin{equation}
    \min_{(\chi,\sigma,\varrho)\in X}L^{\gamma,0}(\chi,\sigma,\varrho)= L^{0}(\chi,\sigma,\varrho)+\dfrac{\gamma}{C_G} \|\chi\|_{TV},
\end{equation}
where $C_G:=\lim_{\epsilon\downarrow0}\dfrac{\|\chi\|_{TV}}{ P_{G_\epsilon}(\chi)/\epsilon}$, and
$$L^{0}(\chi,\sigma,\varrho):=\int_\Omega A(\chi)E_0^{-1}\sigma:\varrho\ dx + \lambda\Big(\int_\Omega A(\chi)E_0^{-1}(|\sigma|^2+|\varrho|^2-|\sigma^*|^2-|\varrho^*|^2)\ dx\Big).$$
Here $(\sigma^*,\varrho^*):=(\dfrac{1}{A(\chi)}E_0\varepsilon(u^*),\dfrac{1}{A(\chi)}E_0\varepsilon(u^*))$ where
$$(u^*,v^*)=\arg\min_{(u,v)\in [H_D^1(\Omega)^d]^2}\Big(\dfrac{1}{2}(a(\chi;u,u)+a(\chi;v,v)-\ell_{\rm in}(u)-\ell_{\rm out}(v)\Big).$$
The functions \( (\sigma,\varrho)\) are independent of $\chi$, while \( (\sigma^*,\varrho^*) \) depend on \( \chi \) (in our application, they are solutions to PDEs).

The continuous problem $P$ is given by
\begin{equation}
    \min_{(\chi,\sigma,\varrho)\in X}L^{\gamma,\epsilon_k}(\chi,\sigma,\varrho)=L^{\epsilon_k}(\chi^{\epsilon_k},\sigma,\varrho) +\dfrac{\gamma}{\epsilon_k} P_{G_{\epsilon_k}}(\chi),
\end{equation}
where \(\chi^{\epsilon_k} := G_{\epsilon_k} * \chi\), $\epsilon_k\rightarrow 0$ as $k\rightarrow +\infty$ and
$$L^{\epsilon_k}(\chi^{\epsilon_k},\sigma,\varrho):=\int_\Omega A(\chi^{\epsilon_k})E_0^{-1}\sigma:\varrho\ dx + \lambda\Big(\int_\Omega A(\chi^{\epsilon_k})E_0^{-1}(|\sigma|^2+|\varrho|^2-|\sigma^*_k|^2-|\varrho^*_k|^2)\ dx\Big).$$
Here $(\sigma^*_k,\varrho^*_k):=(\dfrac{1}{A(\chi^{\epsilon_k})}E_0\varepsilon(u^*_k),\dfrac{1}{A(\chi^{\epsilon_k})}E_0\varepsilon(v^*_k))$ where
$$(u^*_k,v^*_k)=\arg\min_{(u,v)\in [H_D^1(\Omega)^d]^2}\Big(\dfrac{1}{2}(a(\chi^{\epsilon_k};u,u)+a(\chi^{\epsilon_k};v,v)-\ell_{\rm in}(u)-\ell_{\rm out}(v)\Big).$$

The discrete problem $P_n$ is defined as follows: 
\begin{align}
    \min_{(\chi_n,\sigma_n,\varrho_n)\in X_n}L_n^{\gamma,\epsilon_k}(\chi_n,\sigma_n,\varrho_n)=L_n^{\epsilon_k}(\chi_n^{\epsilon_k},\sigma_n,\varrho_n)+\dfrac{\gamma}{\epsilon_k} P_{G_{\epsilon_k}}(\chi_n),
\end{align}
where
$$L_n^{\epsilon_k}(\chi_n^{\epsilon_k},\sigma_n,\varrho_n):=\int_\Omega A(\chi_n^{\epsilon_k})E_0^{-1}\sigma_n:\varrho_n\ dx + \lambda\Big(\int_\Omega A(\chi_n^{\epsilon_k})E_0^{-1}(|\sigma_n|^2+|\varrho_n|^2-|\sigma_n^*|^2-|\varrho_n^*|^2)\ dx\Big).$$
Here $(\sigma^*_n,\varrho^*_n):=(\dfrac{1}{A(\chi^{\epsilon_k})}E_0\varepsilon(u^*_n),\dfrac{1}{A(\chi^{\epsilon_k})}E_0\varepsilon(v^*_n))$ where
$$(u^*_n,v^*_n)=\arg\min_{(u,v)\in V_n\times V_n}\Big(\dfrac{1}{2}(a(\chi^{\epsilon_k};u,u)+a(\chi^{\epsilon_k};v,v)-\ell_{\rm in}(u)-\ell_{\rm out}(v)\Big),$$
$V_n\subset (H_D^1(\Omega))^d$ denotes the $H^1$-conforming finite element space for  displacement. The functions \( \sigma_n^* \) and \( \varrho_n^* \) depend on $\chi_n^{\epsilon_k}$ and thus \( \chi_n \) (in our application, they are FEM solutions of PDEs).

\subsection{Mathematical tools}In this subsection, we collect some known results from the literature which will be used in our following convergence analysis.

\begin{definition}[{\cite[Definition 2.3, Chapter 2, Page 9]{Andrea}}]
    The characterization of \(\Gamma\)-convergence is as follows: 
    \begin{enumerate}
        \item \textbf{Liminf inequality (lower semicontinuity)}: For every $ u \in X $ and every sequence $ \{u_n\} \subset X $ such that $ u_n \to u $ in $ X $,
        $$\liminf_{n \to +\infty} L_n(u_n) \geq L(u).$$
        \item \textbf{Limsup inequality (recovery sequence)}: For every $ u \in X $, there exists a sequence $ \{u_n\} \subset X $ such that $ u_n \to u $ in $ X $ and
        $$\limsup_{n \to +\infty} L_n(u_n) \leq L(u).$$
    \end{enumerate}
\end{definition}
\begin{definition}
    Let \(\{f_n\},f: \mathcal{X\to Y}\). We say that \(f_n\) converges continuously to \(f\), denoted by $f_n\stackrel{c}{\longrightarrow}f$, if for any \(x_n \stackrel{\mathcal{X}}{\to} x\), one has 
    $$\|f_n(x_n) - f(x)\|_\mathcal{Y} \to 0 \quad as\quad n \to \infty.$$
\end{definition}
\begin{definition}
    Let $\{f_n:\mathcal{X}_n\to \mathcal{Y}\}$ for $\mathcal{X}_n\subset \mathcal{X}$ and $f: \mathcal{X\to Y}$. We say that \(f_n\) sub-continuously converges to \(f\), denoted by $f_n\stackrel{sub-c}{\longrightarrow}f$, if for any \(x_n\in\mathcal{X}_n \stackrel{\mathcal{X}}{\longrightarrow} x\), one has 
    $$\|f_n(x_n) - f(x)\|_\mathcal{Y} \to 0 \quad as\quad n \to \infty.$$
    \end{definition}
\begin{lemma}[{\cite[Lemma 1.4, Chapter 1, Page 5]{Hajer}}]\label{Young_inequality}
 For any $f\in L^q(\Omega)$ and $g\in L^r(\Omega)$, we have the following Young's convolution inequality
\[
\|f * g\|_{0,p,\Omega} \leq \|f\|_{0,q,\Omega} \|g\|_{0,r,\Omega},
\]
where \( p, q, r\in [1,+\infty] \) satisfy \( \frac{1}{p} + 1 = \frac{1}{q} + \frac{1}{r} \).
\end{lemma}
\begin{lemma}[{\cite[Theorem 4.22, Chapter 4, Page 109]{Brezis2010}}]
    \label{mollifier_convergence}
    Let \( \rho_\epsilon(x) = \epsilon^{-d} \rho(x/\epsilon) \) be a mollifier in \( \mathbb{R}^d \), where \( \rho \in C_c^\infty(\mathbb{R}^d) \), \( \rho \geq 0 \), \( \int_{\mathbb{R}^d} \rho(x) \, dx = 1 \), and \( \rho \) is radial (e.g., a Gaussian kernel \( \rho(x) = (2\pi)^{-d/2} \exp(-|x|^2/2) \)). For any \( f \in L^p(\mathbb{R}^d) \), \( 1 \leq p < \infty \), the convolution \( \rho_\epsilon * f \) converges to \( f \) in \( L^p(\mathbb{R}^d) \):
    \[
    \|\rho_\epsilon * f - f\|_{L^p(\mathbb{R}^d)} \to 0 \quad \text{as} \quad \epsilon \to 0.
    \]
    This result extends to bounded domains \( \Omega \subset \mathbb{R}^d \) with Lipschitz boundary by zero extension of \( f\) to \( \mathbb{R}^d \).
\end{lemma}
\begin{lemma}[{\cite[Theorem 5.2, Chapter 5, Page 199]{Evans}}]
\label{TV_ineq}
     Suppose $f_k\in BV(\Omega)\ (k=1,2,\cdots)$ and 
    $$ f_k\to f\quad in\ L_{loc}^1(\Omega).$$
    Then we have the following lower semicontinuity of variation measure
    $$\lim_{k\to \infty}\|f_k\|_{TV}\geq \|f\|_{TV}.$$
\end{lemma}

 \begin{lemma}[{\cite[Page 238]{Maz}}]
 \label{parmeter_gamma}
      The approximation \( \dfrac{C_G}{\epsilon}P_{G_\epsilon}(\cdot) \) \(\Gamma\)-converges and  pointwise converges to \(\|\cdot\|_{TV}\). For each $\chi\in BV(\Omega)$, we have (cf. \cite{Maz})
     $$\lim_{\epsilon\downarrow0}\dfrac{C}{\epsilon}P_{G_\epsilon}(\chi)=\|\chi\|_{TV},$$
     where $C_G=\dfrac{2}{\int_\Omega G_1(x)|x_d|\ dx}.$
 \end{lemma}
\begin{definition}
    Let \( (X, \tau) \) be a topological space, and \( \{f_n\} \) be a sequence of functions such that \( f_n: X \to \mathbb{R} \cup \{+\infty\} \). The sequence \( \{f_n\} \) is said to be \emph{equicoercive} if, for every \( t \in \mathbb{R} \), there exists a compact set \( K_t \subset X \) such that, for all \( n \), the sublevel set \( \{ x \in X : f_n(x) \leq t \} \subset K_t \).
\end{definition}

 \begin{lemma}[{\cite[Theorem 5.1, Chapter 5, Page 67]{Andrea}}]
    \label{gamma_convergence}
    Let \( (X, d) \) be a metric space and \( f_n, f: X \to \mathbb{R} \cup \{+\infty\} \) be a sequence of functions such that each \(f_n\) is coercive and lower semicontinuous, and the sequence \(f_n\) \(\Gamma\)-converges to \(f\) and is equicoercive. Suppose \( x^* \in X \) is an isolated local minimizer of \( f \), i.e., there exists a neighborhood \( U \subset X \) of \( x^* \) such that:
    \begin{enumerate}
        \item \( f(x^*) < f(x) \) for all \( x \in U \).
        \item \( x^* \) is the only local minimizer in \( U \).
    \end{enumerate}
    If \( \{x_n\} \subset U \) is a sequence of local minimizers of \( f_n \) such that \( x_n \to x \) in the metric \( d \), then \( x = x^* \).
 \end{lemma}

\subsection{Relationships among the three optimization problems}
In this subsection, we establish the connections among the three optimization problems $Q$, $P$ and $P_n$. The proof strategy is as follows. First, we demonstrate that problem \( P_n \) \(\Gamma\)-converges to problem \( Q \), and discuss the relationships among the three optimization problems. Then, we use Lemma \ref{gamma_convergence} to prove that the local minimizers of \( P_n \) converge to the isolated local minimizers of \( Q \). Finally, we prove that the solution sequence generated by the optimization algorithm for solving \( P_n \) converges to the local minimizers of \( P_n \).

\begin{lemma}
    If the domain $\Omega$ has a Lipschitz boundary \(\Gamma\), then the origin problem \( Q \) has at least one solution.
\end{lemma}
\begin{proof}
    Let \(\{\chi_k\}\subset \mathcal{U}\) be a minimizing sequence for problem \( Q \). Thus, \(\{\chi_k\} \subset BV(\Omega)\). Since \( BV(\Omega) \hookrightarrow L^1(\Omega) \), there exists a convergent subsequence (still denoted by \(\{\chi_h\}\)) that converges to some \(\chi^*\) in \( L^1(\Omega) \). 
    
    Meanwhile, the solutions \((u(\chi_k), v(\chi_k))\) to the state equation \eqref{u_weak} and the adjoint equation \eqref{v_weak} with $\chi$ replaced by $\chi_k$  are bounded in \([H_D^1(\Omega)^d]^2\), so \((u(\chi_h), v(\chi_h))\) weakly converges to some \((u^*, v^*) \in [H_D^1(\Omega)]^2\). According to \cite{Ambr}, we have 
    $$(u^*, v^*) = \arg\min_{(u, v) \in [H_D^1(\Omega)]^2} \Big( \dfrac{1}{2}a(\chi^*;u,u)-\ell_{\rm in}(u)+\dfrac{1}{2}a(\chi^*;v,v)-\ell_{\rm out}(v)\Big).$$
    Define $(\sigma^*,\varrho^*):=(\dfrac{1}{A(\chi^*)}{E}_0\varepsilon(u^*),\dfrac{1}{A(\chi^*)}{E}_0\varepsilon(v^*))$. In the following, we demonstrate that \(\{(\chi^*, \sigma^*, \varrho^*)\}\) is a solution to the problem \( Q \). Note that
    \begin{align*}
        &L^{\gamma,0}(\chi^*,\sigma^*,\varrho^*)-L^{\gamma,0}(\chi_k,\sigma(\chi_k),\varrho(\chi_k))\\
        =&\int_\Omega A(\chi^*)E^{-1}_0\sigma^*:\varrho^*\ dx+\dfrac{\gamma}{C_G} \|\chi^*\|_{TV}-\int_\Omega A(\chi_h)E^{-1}_0\sigma(\chi_k):\varrho(\chi_k)\ dx-\dfrac{\gamma}{C_G} \|\chi_k\|_{TV}\\
        =&\int_\Gamma q_{\rm out}\cdot u^* dx + \dfrac{\gamma}{C_G} \|\chi^*\|_{TV}-\int_\Gamma q_{\rm out}\cdot u(\chi_h) dx-\dfrac{\gamma}{C_G} \|\chi_k\|_{TV}\\
        =&\int_\Gamma q_{\rm out}\cdot (u^*-u(\chi_k))\ ds + \dfrac{\gamma}{C_G}(\|\chi^*\|_{TV}-\|\chi_k\|_{TV}).
    \end{align*}
   Recall the classical trace inequality (cf. \cite{BrennerScott2008})
    \begin{equation}
        \|u\|_{0,\Gamma}\leq C\|u\|_{0,\Omega}^{1/2}\|u\|_{1,\Omega}^{1/2}\qquad \forall u\in H^1(\Omega).
    \end{equation}
    Since \( u(\chi_k) \rightharpoonup u^*\) in $H^1(\Omega)^d$, it follows that \(u(\chi_k) \rightarrow u^*\) in $L^2(\Omega)^d$. Thus, we have
    \begin{align*}
        \int_\Gamma q_{\rm out}\cdot (u^*-u(\chi_k))\ ds&\leq \|q_{\rm out}\|_{0,\Gamma}\|u^*-u(\chi_k)\|_{0,\Gamma}\\
        &\leq C_1\|q_{\rm out}\|_{0,\Gamma}\|u^*-u(\chi_k)\|_{0,\Omega}^{\frac{1}{2}}\|u^*-u(\chi_k)\|_{1,\Omega}^\frac{1}{2}\\
        &\leq C_2\|q_{\rm out}\|_{0,\Gamma}(\|u^*\|_{1,\Omega}+\|u(\chi_k)\|_{1,\Omega})^{\frac{1}{2}}\|u^*-u(\chi_k)\|_{0,\Omega}^{\frac{1}{2}}.
    \end{align*}
    Furthermore, based on the lower semi-continuity of the TV-norm, we obtain
    \begin{align*}
        &\quad\limsup_{h\downarrow 0}\left(L^{\gamma,0}(\chi^*,\sigma^*,\varrho^*)-L^{\gamma,0}(\chi_k,\sigma(\chi_k),\varrho(\chi_k))\right)\\
        &\leq \limsup_{h\downarrow 0}C_2\|q_{\rm out}\|_{0,\Gamma}(\|u^*\|_{1,\Omega}+\|u(\chi_k)\|_{1,\Omega})^{\frac{1}{2}}\|u^*-u(\chi_k)\|_{0,\Omega}^{\frac{1}{2}}+\dfrac{\gamma}{C_G}\limsup_{h\downarrow 0}(\|\chi^*\|_{TV}-\|\chi_k\|_{TV})\\
        &\leq\dfrac{\gamma}{C_G}(\|\chi^*\|_{TV}-\liminf_{h\downarrow 0}\|\chi_k\|_{TV})\leq 0.
    \end{align*}
    Therefore, we obtain the following inequality,
    \begin{equation}
        L^{\gamma,0}(\chi^*,\sigma^*,\varrho^*)\leq \liminf_{h\downarrow 0}L^{\gamma,0}(\chi_h,\sigma(\chi_h),\varrho(\chi_h)).
    \end{equation}
    This shows that \((\chi^*, \sigma^*, \varrho^*)\) is a solution to the optimization problem \( Q \).
\end{proof}

In the following, we first investigate the convergence of problem $P$ to problem $Q$.
\begin{theorem}\label{P_star_Q}
    For each $s>2$, the functional \( L^{\gamma,\epsilon_k} \) \(\Gamma\)-converges to \( L^{\gamma,0} \) in $X^s:=\mathcal{U}\times(S\cap [L^s(\Omega)^{d\times d}]^2)$ with respect to the parameter \(\epsilon_k\), and \( L^{\gamma,\epsilon_k} \) converges pointwise to \( L^{\gamma,0} \) in $X$.
\end{theorem}
\begin{proof}
    It follows from the definition that
    $$L^{\gamma,\epsilon_k}(\chi,\sigma,\varrho)=L^{\epsilon_k}(\chi^{\epsilon_k},\sigma,\varrho)+\dfrac{\gamma}{\epsilon_k}P_{G_{\epsilon_k}}(\chi),$$
    $$L^{\gamma,0}(\chi,\sigma,\varrho)=L^0(\chi,\sigma,\varrho)+\dfrac{\gamma}{C_G}\|\chi\|_{TV},$$
    so we separately discuss the convergence of these two parts. By Lemma \ref{parmeter_gamma}, we have established that \( \dfrac{\gamma}{\epsilon_k}P_{G_{\epsilon_k}}(\cdot) \) \(\Gamma\)-converges and pointwise converges to \(\dfrac{\gamma}{C_G}\|\cdot\|_{TV}\). For $L^{\epsilon_k}(\cdot)$ and $L^0(\cdot)$, we will prove the continuous convergence, denoted as $L^{\epsilon_k}\stackrel{c}{\longrightarrow}L^0 $.

    Let $\chi_m\stackrel{L^1}{\longrightarrow}\chi$ and $(\sigma_m,\varrho_m)\stackrel{L^s}{\longrightarrow}(\sigma,\varrho)$ for some  $s>2$, it follows from the definition and \eqref{G_rep}  that
    \begin{align*}
        L^{\epsilon_k}(\chi_m^{\epsilon_k},\sigma_m,\varrho_m)-L^0(\chi,\sigma,\varrho)
        &=\int_\Omega(A(\chi_m^{\epsilon_k})-A(\chi))E_0^{-1}\sigma_m:\varrho_m\ dx + \int_\Omega A(\chi)E_0^{-1}(\sigma_m-\sigma):\varrho_m\ dx \\
        &+\int_\Omega A(\chi)E_0^{-1}\sigma:(\varrho_m-\varrho) dx + \lambda\int_\Omega(A(\chi_m^{\epsilon_k})-A(\chi))E_0^{-1}(|\sigma_m|^2+|\varrho_m|^2) dx\\
        &+\lambda\int_\Omega A(\chi)E_0^{-1}(|\sigma_m|^2-|\sigma|^2+|\varrho_m|^2-|\varrho|^2)\ dx + \lambda a(\chi;u(\chi),u(\chi))\\
        &+\lambda a(\chi;v(\chi),v(\chi))-\lambda a(\chi_m^{\epsilon_k};u(\chi_m^{\epsilon_k}),u(\chi_m^{\epsilon_k}))
        -\lambda a(\chi_m^{\epsilon_k};v(\chi_m^{\epsilon_k}),v(\chi_m^{\epsilon_k})).
    \end{align*}
    We first address the term involving \(\chi_m^{\epsilon_k} - \chi\). Note that for elements in \(\mathcal{U}\), \(L^1\)-convergence is equivalent to \(L^p\)-convergence  for any \(p \in [1, +\infty)\). We take 
    $$ \int_\Omega(A(\chi_m^{\epsilon_k})-A(\chi))E_0^{-1}\sigma_m:\varrho_m\ dx$$
    as an example to illustrate the main idea of the proof. Recalling the definition of $A(\chi)$, we have
    \begin{align*}
        &\int_\Omega(A(\chi_m^{\epsilon_k})-A(\chi))E_0^{-1}\sigma_m:\varrho_m\ dx\\
        &=\int_\Omega(A(\chi_m^{\epsilon_k})-A(\chi^{\epsilon_k}))E_0^{-1}\sigma_m:\varrho_m\ dx+\int_\Omega(A(\chi^{\epsilon_k})-A(\chi))E_0^{-1}\sigma_m:\varrho_m\ dx\\
        &=\Big(\dfrac{1}{E_{\max}}-\dfrac{1}{E_{\min}}\Big)\Big(\int_\Omega G_{\epsilon_k}*(\chi_m-\chi)E_0^{-1}\sigma_m:\varrho_m\ dx + \int_\Omega(G_{\epsilon_k}*\chi-\chi)E_0^{-1}\sigma_m:\varrho_m\ dx\Big)\\
        &\leq \dfrac{C}{E_{\min}}\Big(\|G_{\epsilon_k}*(\chi_m-\chi)\|_{0,\frac{s}{s-2},\Omega}+\|G_{\epsilon_k}*\chi-\chi\|_{0,\frac{s}{s-2},\Omega}\Big)\|\sigma_m\|_{0,s,\Omega}\|\varrho_m\|_{0,s,\Omega}\\
        &\leq\dfrac{C}{E_{\min}}\Big(\|\chi_m-\chi\|_{0,\frac{s}{s-2},\Omega}+\|G_{\epsilon_k}*\chi-\chi\|_{0,\frac{s}{s-2},\Omega}\Big)\|\sigma_m\|_{0,s,\Omega}\|\varrho_m\|_{0,s,\Omega}\to 0 \quad as \quad k,m\to+\infty,
    \end{align*}
    where we used the generalized Hölder inequality in the first inequality and Young's inequality for convolutions (cf. Lemma \ref{Young_inequality}) in the second inequality. The convergence of the first term in the limit is due to \(\chi_k\stackrel{L^p}{\longrightarrow} \chi\), while the second term converges because of the convergence of the  smooth Dirac sequence (cf. Lemma \ref{mollifier_convergence}).

    The handling of \(\sigma_m - \sigma\) and \(\varrho_m - \varrho\) is straightforward. The key lies in the estimate of \(u(\chi_m^{\epsilon_k}) - u(\chi)\), which involves the stability of solutions to linear elasticity equations with respect to the elastic tensor.
    \begin{align*}
        a(\chi;u(\chi),u(\chi))-a(\chi_m^{\epsilon_k};u(\chi_m^{\epsilon_k}),u(\chi_m^{\epsilon_k}))
        &=\ell_{\rm in}(u(\chi))-\ell_{\rm in}(u(\chi_m^{\epsilon_k}))\\
        &=\|\ell_{\rm in}\|_{0,\Gamma_N}\|u(\chi)-u(\chi_m^{\epsilon_k})\|_{0,\Gamma}\\
        &\leq C\|\ell_{\rm in}\|_{0,\Gamma_N}\|u(\chi)-u(\chi_m^{\epsilon_k})\|_{1,\Omega}.
    \end{align*}
    On the other hand, we have
    \begin{align*}
        a(\chi;u(\chi),\hat{u})=\ell_{\rm in}(\hat{u})\quad \forall \hat{u}\in H_D^1(\Omega)^d,\quad 
        a(\chi_m^{\epsilon_k};u(\chi_m^{\epsilon_k}),\tilde{u})=\ell_{\rm in}(\tilde{u})\quad \forall\tilde{u}\in  H_D^1(\Omega)^d.
    \end{align*}
    Setting $\hat{u}=\tilde{u}=u(\chi)-u(\chi_k^{\epsilon_k})$, we have
    \begin{align*}
        a(\chi;u(\chi)-u(\chi_m^{\epsilon_k}),u(\chi)-u(\chi_m^{\epsilon_k}))
    &=a(\chi_m^{\epsilon_k};u(\chi_m^{\epsilon_k}),u(\chi)-u(\chi_m^{\epsilon_k}))-a(\chi;u(\chi_m^{\epsilon_k}),u(\chi)-u(\chi_m^{\epsilon_k}))\\
       & =\Big(\dfrac{1}{E_{\max}}-\dfrac{1}{E_{\min}}\Big)\int_\Omega (G_{\epsilon_k}*\chi_m-\chi)E_0\varepsilon(u(\chi_m^{\epsilon_k})):\varepsilon(u(\chi)-u(\chi_m^{\epsilon_k}))dx\\
        &\leq\dfrac{C}{E_{\min}}\|G_{\epsilon_k}*\chi_m-\chi\|_{0,\frac{2s}{s-2},\Omega}\|u(\chi_m^{\epsilon_k})\|_{1,s,\Omega}\|u(\chi)-u(\chi_m^{\epsilon_k})\|_{1,\Omega}.
    \end{align*}
    Hence, 
    $$\|u(\chi)-u(\chi_m^{\epsilon_k})\|_{1,\Omega}\leq C\dfrac{E_{\max}}{E_{\min}}\|G_{\epsilon_k}*\chi_m-\chi\|_{0,\frac{2s}{s-2},\Omega}\|u(\chi_m^{\epsilon_k})\|_{1,s,\Omega}.$$
    The remaining terms can also be handled sequentially, ultimately controlled by \(\|\chi_m - \chi\|_{0,p,\Omega}\), \(\|G_{\epsilon_k} * \chi - \chi\|_{0,p,\Omega}\), \(\|\sigma_m - \sigma\|_{0,s,\Omega}\), and \(\|\varrho_m - \varrho\|_{0,s,\Omega}\). This completes the proof that \(L^{{\epsilon_k}} \stackrel{c}{\longrightarrow} L^{0}\).
\end{proof}

Then we investigate the convergence of problem $P_n$ to problem $P$.
\begin{theorem}\label{P_n_star_P}
     For all $\epsilon_k> 0$, we have $L_n^{\gamma,{\epsilon_k}}\stackrel{sub-c}{\longrightarrow} L^{\gamma,{\epsilon_k}}$ and  $L_n^{\epsilon_k}\stackrel{sub-c}{\longrightarrow} L^{\epsilon_k}$ as $n\rightarrow +\infty$. In particular, if the solutions to the linear elasticity equation have a higher regularity \((u, v) \in [H^s(\Omega)^d\cap H_D^1(\Omega)^d]^2\) for some \(s>1\)), then for any sequence \(\{\epsilon_n\}_{n \geq 0}\) with \(\epsilon_n \to 0\), we have \(L_n^{\epsilon_n} \stackrel{sub-c}{\longrightarrow}  L^0\).
\end{theorem}
\begin{proof}
    Let $\chi_n\stackrel{L^1}{\longrightarrow}\chi$ and $(\sigma_n,\varrho_n)\stackrel{L^2}{\longrightarrow}(\sigma,\varrho)$. We first examine the continuous convergence of the perimeter regularization.
    Since $\int_\Omega\chi_n dx\leq |\Omega|<\infty$, we have \cite[eq. (1.2), Page 236]{Maz} 
    \begin{equation}
        P_{G_{\epsilon_k}}(\chi_n)=\int_\Omega \chi_n\ dx-\int_\Omega (G_{\epsilon_k}*\chi_n)(x)\chi_n(x)\ dx, 
    \end{equation}
    then
    \begin{align*}
        &|P_{G_{\epsilon_k}}(\chi_n) - P_{G_{\epsilon_k}}(\chi)|\\
        &\leq \|\chi_n-\chi\|_{0,1,\Omega}+\int_\Omega|(G_{\epsilon_k}*\chi_n)\chi_n-(G_{\epsilon_k}*\chi)\chi|\ dx\\
        &\leq \|\chi_n-\chi\|_{0,1,\Omega}+\int_\Omega|(G_{\epsilon_k}*\chi_n)\chi_n-(G_{\epsilon_k}*\chi_n)\chi|\ dx+\int_\Omega|(G_{\epsilon_k}*\chi_n)\chi-(G_{\epsilon_k}*\chi)\chi|\ dx\\
        &\leq \|\chi_n-\chi\|_{0,1,\Omega}+\int_\Omega(G_{\epsilon_k}*\chi_n)|\chi_n-\chi|\ dx+\int_\Omega|G_{\epsilon_k}*(\chi_n-\chi)|\chi\ dx\\
        &\leq 3\|\chi_n-\chi\|_{0,1,\Omega},
    \end{align*}
   where in the last inequality, we used Young's convolution inequality with $p=q=r=1$ (cf. Lemma \ref{Young_inequality}).
    
The following proof is very similar to that of Theorem \ref{P_star_Q}, and even simpler, as we do not require \((\sigma, \varrho) \in [L^s(\Omega)^{d \times d}]^2\). Using again \eqref{G_rep} and the definition, we have the following expression
    \begin{align*}
        &L_n^{\epsilon_k}(\chi_n^{\epsilon_k},\sigma_n,\varrho_n)-L^{\epsilon_k}(\chi^{\epsilon_k},\sigma,\varrho)\\
        =& \int_\Omega (A(\chi_n^{\epsilon_k}) - A(\chi^{\epsilon_k})) E_0^{-1} \sigma_n : \varrho_n \, dx + \int_\Omega A(\chi^{\epsilon_k}) E_0^{-1} (\sigma_n - \sigma) : \varrho_n \, dx \\
        & + \int_\Omega A(\chi^{\epsilon_k}) E_0^{-1} \sigma : (\varrho_n - \varrho) \, dx + \lambda \int_\Omega (A(\chi_n^{\epsilon_k}) - A(\chi^{\epsilon_k})) E_0^{-1} (|\sigma_n|^2 + |\varrho_n|^2) \, dx \\
        & + \lambda \int_\Omega A(\chi^{\epsilon_k}) E_0^{-1} (|\sigma_n|^2 - |\sigma|^2 + |\varrho_n|^2 - |\varrho|^2) \, dx + \lambda a(\chi^{\epsilon_k}; u(\chi^{\epsilon_k}), u(\chi^{\epsilon_k})) \\
        & + \lambda a(\chi^{\epsilon_k}; v(\chi^{\epsilon_k}), v(\chi^{\epsilon_k})) - \lambda a(\chi_n^{\epsilon_k}; u_n(\chi_n^{\epsilon_k}), u_n(\chi_n^{\epsilon_k})) - \lambda a(\chi_n^{\epsilon_k}; v_n(\chi_n^{\epsilon_k}), v_n(\chi_n^{\epsilon_k})),
    \end{align*}
    which can be easily decomposed into terms involving \(\chi_n - \chi\), \(\sigma_n - \sigma\), \(\varrho_n - \varrho\), \(u_n - u(\chi)\), and \(v_n - v(\chi)\). The estimates for terms containing \(\sigma_n - \sigma\) and \(\varrho_n - \varrho\) are straightforward. 
    
We first estimate the terms containing \(\chi_n - \chi\) as follows
    \begin{align*}
        \int_\Omega (A(\chi_n^{\epsilon_k}) - A(\chi^{\epsilon_k})) E_0^{-1} \sigma_n : \varrho_n \, dx
        &= \Big( \dfrac{1}{E_{\max}} - \dfrac{1}{E_{\min}}\Big) \int_\Omega G_{\epsilon_k} * (\chi_n - \chi) E_0^{-1} \sigma_n : \varrho_n \, dx \\
        &\leq  \dfrac{C}{E_{\min}} \| G_{\epsilon_k} * (\chi_n - \chi) \|_{0,\infty,\Omega} \| \sigma_n \|_{0,\Omega} \| \varrho_n \|_{0,\Omega} \\
        &\leq  \dfrac{C}{E_{\min}} \| G_{\epsilon_k} \|_{0,\Omega} \| \chi_n - \chi \|_{0,\Omega} \| \sigma_n \|_{0,\Omega} \| \varrho_n \|_{0,\Omega} \to 0 \quad \text{as} \quad n \to +\infty,
    \end{align*}
   where we employed Young's convolution inequality in the last inequality. The boundedness of \(\|G_{\epsilon_k}\|_{0,\Omega}\) is guaranteed because \({\epsilon_k}\) is fixed at this point. 
   
Then we estimate the terms containing \(u_n(\chi_n^{\epsilon_k}) - u(\chi^{\epsilon_k})\) below
    \begin{align*}
        &a(\chi^{\epsilon_k};u(\chi^{\epsilon_k}),u(\chi^{\epsilon_k}))-a(\chi_n^{\epsilon_k};u_n(\chi_n^{\epsilon_k}),u_n(\chi_n^{\epsilon_k}))\nonumber\\
        &=\ell_{\rm in}(u(\chi^{\epsilon_k}))-\ell_{\rm in}(u_n(\chi_n^{\epsilon_k}))\\
        &=\ell_{\rm in}(u(\chi^{\epsilon_k})-u_n(\chi^{\epsilon_k}))+\ell_{\rm in}(u_n(\chi^{\epsilon_k})-u_n(\chi_n^{\epsilon_k}))\\
        &\leq \|\ell_{\rm in}\|_{0,\Gamma_N}\|u(\chi^{\epsilon_k})-u_n(\chi^{\epsilon_k})\|_{0,\Gamma_N}+\|\ell_{\rm in}\|_{0,\Gamma_N}\|u_n(\chi^{\epsilon_k})-u_n(\chi_n^{\epsilon_k})\|_{0,\Gamma_N}\\
        &\leq C\|\ell_{\rm in}\|_{0,\Gamma_N}\big [\|u(\chi^{\epsilon_k})-u_n(\chi^{\epsilon_k})\|_{1,\Omega}+\|u_n(\chi^{\epsilon_k})-u_n(\chi_n^{\epsilon_k})\|_{1,\Omega}\big ].
    \end{align*}
    The convergence of the first term can be obtained by using classical finite element error estimates
    $$\|u(\chi^{\epsilon_k})-u_n(\chi^{\epsilon_k})\|_{1,\Omega}\leq Ch_n^{s-1}\|u(\chi^{\epsilon_k})\|_{s,\Omega}\leq Ch_n^{s-1}\|\ell_{\rm in}\|_{0,\Gamma_N}.$$
    On the other hand, we have
    \begin{align*}
        a(\chi^{\epsilon_k};u_n(\chi^{\epsilon_k}),\hat{u}_n)=\ell_{\rm in}(\hat{u}_n)\quad \forall \hat{u}_n\in V_n,\qquad
        a(\chi_n^{\epsilon_k};u_n(\chi_n^{\epsilon_k}),\tilde{u}_n)=\ell_{\rm in}(\tilde{u}_n)\quad \forall\tilde{u}_n\in  V_n,
    \end{align*}
    where $V_n\subset H_D^1(\Omega)^d$ is the finite element space.
    
Setting $\hat{u}_n=\tilde{u}_n=u_n(\chi^{\epsilon_k})-u_n(\chi_n^{\epsilon_k})$ we obtain
    \begin{align*}
        &a(\chi^{\epsilon_k};u_n(\chi^{\epsilon_k})-u_n(\chi_n^{\epsilon_k}),u_n(\chi^{\epsilon_k})-u_n(\chi_n^{\epsilon_k}))\\
         &= a(\chi_n^{\epsilon_k}; u_n(\chi_n^{\epsilon_k}), u_n(\chi^{\epsilon_k}) - u_n(\chi_n^{\epsilon_k})) - a(\chi^{\epsilon_k}; u_n(\chi_n^{\epsilon_k}), u_n(\chi^{\epsilon_k}) - u_n(\chi_n^{\epsilon_k})) \\
        &= \Big( \dfrac{1}{E_{\max}} - \dfrac{1}{E_{\min}} \Big) \int_\Omega G_{\epsilon_k} * (\chi_n - \chi) E_0 \varepsilon(u_n(\chi_n^{\epsilon_k})) : \varepsilon(u_n(\chi^{\epsilon_k}) - u_n(\chi_n^{\epsilon_k})) \, dx \\
        &\leq  \dfrac{C}{E_{\min}} \| G_{\epsilon_k} * (\chi_n - \chi) \|_{0,\infty,\Omega} \|u_n(\chi_n^{\epsilon_k})\|_{1,\Omega} \|u_n(\chi^{\epsilon_k}) - u_n(\chi_n^{\epsilon_k})\|_{1,\Omega} \\
        &\leq  \dfrac{C}{E_{\min}} \| G_{\epsilon_k} \|_{0,\Omega} \| \chi_n - \chi \|_{0,\Omega} \|u_n(\chi_n^{\epsilon_k})\|_{1,\Omega} \|u_n(\chi^{\epsilon_k}) - u_n(\chi_n^{\epsilon_k})\|_{1,\Omega}.
    \end{align*}
    Hence, 
    $$\|u_n(\chi^{\epsilon_k})-u_n(\chi_n^{\epsilon_k})\|_{1,\Omega}\leq C\dfrac{E_{\max}}{E_{\min}}\|G_{\epsilon_k}\|_{0,\Omega}\|\chi_n-\chi\|_{0,\Omega}\|u_n(\chi_n^{\epsilon_k})\|_{1,\Omega}.$$
    The remaining terms can also be handled sequentially.
    
    It is easy to see that among all the estimates above, only the finite element error estimate depends on the value of \(\epsilon_k\). When \(u \in H^s(\Omega)^d\cap H_D^1(\Omega)^d\), the finite element error estimate has a convergence order of \(s-1\), depending only on the mesh size \(h_n\). At this point, the estimates will ultimately be controlled by \(\|\chi_n - \chi\|_{0,\Omega}\), \(\|\sigma_n - \sigma\|_{0,\Omega}\), and \(\|\varrho_n - \varrho\|_{0,\Omega}\), we obtain the following
    \begin{eqnarray} 
    L_n^{\gamma,{\epsilon_k}}\stackrel{sub-c}{\longrightarrow} L^{\gamma,{\epsilon_k}}, \qquad L_n^{\epsilon_k}\stackrel{sub-c}{\longrightarrow} L^{\epsilon_k}\quad\mbox{for\ any}\ \epsilon_k\geq 0\ \mbox{as}\ n\rightarrow +\infty.\nonumber
    \end{eqnarray}
    If \(u \in H^s(\Omega)^d\cap H_D^1(\Omega)^d\), then
    \begin{eqnarray} 
    L_n^{\epsilon_n}\stackrel{sub-c}{\longrightarrow} L^{0}\quad\mbox{as}\  n\rightarrow +\infty.\nonumber
    \end{eqnarray}
    This finishes the proof.
\end{proof}
We note that the requirement \(u \in H^s(\Omega)^d\cap H_D^1(\Omega)^d\) in Theorem \ref{P_n_star_P} is not indispensable. In fact, we can show the strong convergence of $u_n(\chi^{\epsilon_k})$ to $u(\chi^{\epsilon_k})$ as $h_n\rightarrow 0$ by using density arguments.

Now we are ready to prove the convergence of problem $P_n$ to problem $Q$.
\begin{theorem}\label{P_n_star_Q}
    For each $s>2$ and $N>0$, assume the sequence $\{X_n\}$ is nested such that $X_1\subset X_2\subset\cdots\subset X_N$. Then, for any sequence $ \{\epsilon_n\}_{n \geq 1} $, with $ \epsilon_n \to 0 $, we have the functional \( L_n^{\gamma,{\epsilon_n}} \) \(\Gamma\)-converges to \( L^{\gamma,0} \) in $X^s$ with respect to  parameters $n$, while \( L_n^{\gamma,{\epsilon_n}} \) converges pointwise to \( L^{\gamma,0} \) in $\bigcup_{n\leq N} X_n$.
\end{theorem}
\begin{proof}
    We first note that the domains of \( L_n^{\gamma,{\epsilon_k}} \) and \( L^{\gamma,0} \) are different. However, since we use the conforming finite element method, we have \( X_n \subset X \). Thus, we can naturally extend the definition of \( L_n \) to \( X \), and then discuss \(\Gamma\)-convergence and pointwise convergence without ambiguity. 
    
    We consider a larger family of functions, where the subscripts and superscripts are denoted by \(\epsilon_k\) and \(n\), respectively. For any $\epsilon_k$ and $n$, we extend the definition as follows:
    \begin{equation}\label{bar_P}
        \bar L_n^{\gamma,{\epsilon_k}}(\chi,\sigma,\varrho)=
        \begin{cases} 
             L_n^{\gamma,{\epsilon_k}}(\chi,\sigma,\varrho) & \text{if } (\chi,\sigma,\varrho)\in X_n, \\
            +\infty & \text{if } (\chi,\sigma,\varrho)\notin X_n.
        \end{cases}
    \end{equation}
    Next, we proceed to prove \(\Gamma\)-convergence. Given \( (\chi, \sigma, \varrho) \in X^s \), we have the following conclusions
    \begin{enumerate}
        \item For all $(\chi_n, \sigma_n, \varrho_n){\rightarrow}(\chi, \sigma, \varrho)$ in the space $(L^1(\Omega),L^s(\Omega),L^s(\Omega))$, we have
        \begin{align*}
            &\bar L_n^{\gamma,\epsilon_n}(\chi_n, \sigma_n, \varrho_n)-L^{\gamma,0}(\chi, \sigma, \varrho)\\
            &=[\bar L_n^{\epsilon_n}(\chi_n, \sigma_n, \varrho_n)-L^0(\chi, \sigma, \varrho)]+\left[\dfrac{\gamma}{\epsilon_n}P_{G_{\epsilon_n}}(\chi_n)-\dfrac{\gamma}{C_G}\|\chi\|_{TV}\right].
        \end{align*}
        As $n\rightarrow +\infty$, the first term is known to be greater than or converges to 0 by Theorem \ref{P_n_star_P} (sub-continuous convergence of \( L_n^{\epsilon_n} \)), while the second term is greater than 0 by Lemma \ref{parmeter_gamma} (lower semicontinuity by the \(\Gamma\)-convergence of the nonlocal perimeter approximation). Therefore, we ultimately have 
        \begin{equation}
            \liminf_{n\to+\infty}\bar L_n^{\gamma,\epsilon_n}(\chi_n, \sigma_n, \varrho_n)\geq L^{\gamma,0}(\chi, \sigma, \varrho).\nonumber
        \end{equation} 
        
        \item With the \(\Gamma\)-convergence of \( L^{\gamma,\epsilon_k} \) to \( L^{\gamma,0} \) established in Theorem \ref{P_star_Q}, we know that there exists a sequence \(\{(\chi_k, \sigma_k, \varrho_k)\}_{k\geq 1} \subset X^s\) such that
        \begin{equation}
        \limsup_{k\to\infty}L^{\gamma,\epsilon_k}(\chi_k, \sigma_k, \varrho_k)\leq L^{\gamma,0}(\chi, \sigma, \varrho).\nonumber
        \end{equation} 
        By the definition of sub-continuous convergence, for each \(k\) and any \(\{(\chi_{k,n}, \sigma_{k,n}, \varrho_{k,n})\} \subset X_n\subset X\), we have
        \begin{equation}
        \lim_{n\to\infty}L_n^{\gamma,\epsilon_k}(\chi_{k,n}, \sigma_{k,n}, \varrho_{k,n})= L^{\gamma,\epsilon_k}(\chi_k, \sigma_k, \varrho_k).\nonumber
        \end{equation} 
        Thus, we take the recovery sequence \(\{(\chi_{k,k}, \sigma_{k,k}, \varrho_{k,k})\}\) and obtain 
         \begin{equation}
         \limsup_{k\to\infty}\bar L_k^{\gamma,\epsilon_k}(\chi_{k,k}, \sigma_{k,k}, \varrho_{k,k})=\limsup_{k\to\infty}L_k^{\gamma,\epsilon_k}(\chi_{k,k}, \sigma_{k,k}, \varrho_{k,k})\leq L^{\gamma,0}(\chi, \sigma, \varrho).
         \nonumber
        \end{equation} 
    \end{enumerate}
    Recall Definition 3.1, combining (1) and (2) we prove $\bar L_n^{\gamma,\epsilon_n}\stackrel{\Gamma}{\longrightarrow} L^{\gamma,0}$ in $X^s$.

    In the following, we consider pointwise convergence. For any \((\chi_{N},\sigma_{N},\varrho_{N}) \in \bigcup_{n' \leq N} X_{n'}\), as long as \( n > N \), we have \((\chi_{N},\sigma_{N},\varrho_{N}) \in X_{n}\). At this point, \(\bar{L}_n^{\gamma,\epsilon_n}(\chi_{N},\sigma_{N},\varrho_{N}) = L_n^{\gamma,\epsilon_n}(\chi_{N},\sigma_{N},\varrho_{N})\). Therefore, it suffices to use the pointwise convergence established in Lemma \ref{parmeter_gamma} and Theorem \ref{P_n_star_P} to conclude 
    \begin{align*}
        \lim_{n\to\infty}\bar{L}_n^{\gamma,\epsilon_n}(\chi_{N},\sigma_{N},\varrho_{N}) = \lim_{n\to\infty}\Big(L_n^{\epsilon_n}(\chi_{N}^{\epsilon_n},\sigma_{N},\varrho_{N})+\dfrac{\gamma}{C_G}P_{G_{\epsilon_n}}(\chi_N)\Big)=L^{\gamma,0}(\chi_{N},\sigma_{N},\varrho_{N}).
    \end{align*}
    This finishes the proof.
\end{proof}

We note that in Theorem \ref{P_n_star_Q}, we actually proved the convergence of \(\bar{L}_n^{\gamma,\epsilon_k}\) to \(L^{\gamma,0}\). The reason we say that the convergence of \(L_n^{\gamma,\epsilon_k}\) to \(L^{\gamma,0}\) is that, for our research subject—the convergence between optimization problems—the minimizers of \(\bar{L}_n^{\gamma,\epsilon_k}\) and \(L_n^{\gamma,\epsilon_k}\) are completely equivalent. Henceforth, we no longer distinguish between \(\bar{L}_n^{\gamma,\epsilon_k}\) and \(L_n^{\gamma,\epsilon_k}\), and we directly treat \(L_n^{\gamma,\epsilon_k}\) as defined on \(X\).

We can illustrate the relationships among the three problems $Q$, $P$ and $P_n$ using the following commute diagram:
\begin{center}
    \begin{tikzcd}
    L^{\gamma,\epsilon_k} \arrow[r, "k"] &  L^{\gamma,0} \\
    L_n^{\gamma,\epsilon_k} \arrow[u, "n"] \arrow[ur, "n\ k"']
    \end{tikzcd}
\end{center}

\subsection{Convergence characterization}

From the above commute diagram, we can observe two paths for the convergence of problem $   P_n   $ to problem $   Q   $, where $  \epsilon_k  $ and $h_n$ are independent in the limit process. However, when considering the convergence of local minimizers, we can no longer ignore this issue, as the existence of a solution to problem $P$ is not yet guaranteed. Therefore, when solving the problem $Q$, we cannot first refine the mesh and then adjust the perimeter approximation parameter $  \epsilon_k  $, as this cannot ensure the convergence of minimizers. We first prove the following convergence theorem for local minimizers. 
\begin{theorem}
     Let $\{\chi_n\}$ be a sequence of local minimizers of the problem $\{P_n\}$ with an appropriate chosen $  \epsilon_n  $ for each $n$. For every accumulation point $\chi^+$ of $\{\chi_n\}  $, if $  \chi^+ \in U$ with $U$ the minimal neighborhood (in the $   L^1   $-metric) of some isolated local minimizer $\chi^*$ of problem $Q$, then $\chi^+ = \chi^*$.
\end{theorem}
\begin{proof}
    It suffices to verify the conditions of Lemma \ref{gamma_convergence}. We have already proved that $L_n^{\gamma,\epsilon_n}$ $ \Gamma  $-converges to $L^{\gamma,0}$, while the coercivity of $L_n^{\gamma,\epsilon_n} $ (in the $L^1$-metric) and its lower semicontinuity are evident. The key step now is to verify the equicoercivity. 
    
    We first prove that \( L_n^{\epsilon_n} \) has a uniform lower bound \( M \). In Section 2, we have shown that \(L_n^{\epsilon_n}(\chi_n^{\epsilon_n}, \sigma_n^*, \varrho_n^*)=\min_{(\hat\sigma_n, \hat\varrho_n) \in S_n} L_n^{\epsilon_n}(\chi_n^{\epsilon_n}, \hat\sigma_n, \hat\varrho_n) \).  Thus, 
    \begin{align*}
        L_n^{\epsilon_n}(\chi_n^{\epsilon_n}, \hat\sigma_n, \hat\varrho_n)&\geq L_n^{\epsilon_n}(\chi_n^{\epsilon_n}, \sigma_n^*, \varrho_n^*)\\
        &=a(\chi_n^{\epsilon_n};u_n^*,v_n^*)\\
        &=\ell_{\rm out}(u_n^*)\\
        &\geq -C_1\|\ell_{\rm out}\|_{0,\Gamma_N}\|u_n^*\|_{1,\Omega}\\
        &\geq -C_2 \|\ell_{\rm out}\|_{0,\Gamma_N}\|\ell_{\rm in}\|_{0,\Gamma_N}.
    \end{align*}
    Therefore, \( M := -C_2 \|\ell_{\rm out}\|_{0,\Gamma_N}\|\ell_{\rm in}\|_{0,\Gamma_N} \leq L_n^{\epsilon_n}(\chi_n^{\epsilon_n}, \sigma_n, \varrho_n) \) for all $n$, $\epsilon_n$ and  $(\chi_n, \sigma_n, \varrho_n) \in X$.
    
    Let $t \in \mathbb{R}$, we choose the set $$   K_t = \Big\{ \chi \in \mathcal{U} : \dfrac{\gamma}{C_G}\|\chi\|_{TV} \leq \max(t-M+1,1)\Big\}.$$ 
    Since $BV(\Omega)$ is compactly embedded in $L^1(\Omega)$ and $K_t$ is bounded in $BV(\Omega)$, it is a compact set in $L^1(\Omega)$. For any $n$, we note that $\mathcal{U}_n $ is a finite set and Lemma \ref{gamma_convergence} have
        $$\lim_{\epsilon\downarrow 0 }\dfrac{\gamma}{\epsilon} P_{G_{\epsilon}}(\chi_{n,i}) = \dfrac{\gamma}{C_G}\|\chi_{n,i}\|_{TV}\quad \forall \chi_{n,i}\in \mathcal{U}_n.$$
    Hence for each $\delta>0$, fix $\chi_{n,i}\in \mathcal{U}_n$, there exists a sufficiently small $\epsilon_{n,i}'>0$, for each $\epsilon_{n,i}\in (0,\epsilon_{n,i}')$ have
    $$\Big|\dfrac{\gamma}{\epsilon_{n,i}} P_{G_{\epsilon_{n,i}}}(\chi_{n,i}) - \dfrac{\gamma}{C_G}\|\chi_{n,i}\|_{TV}\Big| <\delta. $$
    We set $\delta = 1,\ \epsilon_n = \min_i \epsilon_{n,i}$, if \((\chi_n,\sigma_n,\rho_n)\in X_n\) such that
    $$L_n^{\gamma,\epsilon_n}(\chi_n,\sigma_n,\rho_n)\leq t,$$
    then 
    $$ \dfrac{\gamma}{C_G}\|\chi_n\|_{TV}\leq \dfrac{\gamma}{\epsilon_n}P_{G_{\epsilon_n}}(\chi_n)+1\leq t-L_n^{\epsilon_n}(\chi_n^{\epsilon_n},\sigma_n,\rho_n)+1\leq t-M+1.
    $$
    Hence, $\{\chi_n: L_n^{\gamma,\epsilon_n}(\chi_n,\sigma_n,\rho_n)<t\}\subset K_t$ for any $n\geq 1$, i.e., the sequence $\{L_n^{\gamma,\epsilon_n}\}$ satisfies the equicoercivity. We thus obtain the conclusion of the theorem. 
\end{proof}

%\newpage

\section{Penalty method for the heat transfer problems}\label{sec:heat}
In this section, we extend the proposed penalty method to heat transfer problems. 
\subsection{Heat transfer problems}
We use the same setting for the domain $\Omega$ as in Section 2. Consider the following topology optimization problem
\begin{align}
\min_{(\chi,T)\in\mathcal{U}\times H^1_g(\Omega)}J(\chi,T) = \int_{\Omega}q(\chi)Tdx
\end{align}
subject to 
\begin{align}\label{heat:state}
\int_{\Omega}\kappa(\chi)\nabla T\cdot\nabla Sdx = \int_{\Omega}q(\chi)Sdx\quad \forall S\in H_D^1(\Omega):=\{v\in H^1(\Omega):\ v=0\ \mbox{on}\ \Gamma_D\},
\end{align}
where $H^1_g(\Omega):=\{v\in H^1(\Omega):\ v=g\ \mbox{on}\ \Gamma_D\}$ for some $g\in H^{1\over 2}(\Gamma_D)$ and $\chi$ is an indicator function of the shape. The admissible set of feasible shapes is given by 
\begin{eqnarray}
\mathcal{U} =\Big\{\chi\in BV(\Omega):\ \chi\in \{0,1\},\ \int_{\Omega}\chi dx \leq\beta|\Omega|\Big\}.\nonumber
\end{eqnarray}
In this model, the heat source is design-dependent, i.e., 
\begin{eqnarray}
q(\chi)=q_1\chi + q_2(1-\chi)\quad (q_1<q_2)\nonumber
\end{eqnarray}
for some positive constants $q_1,q_2$, while the heat conductivity coefficient is also design-dependent
\begin{eqnarray}
\kappa(\chi) = \kappa_1\chi+ \kappa_2(1-\chi)\quad(\kappa_1>\kappa_2)\nonumber
\end{eqnarray}
for some positive constants $\kappa_1,\kappa_2$. Without loss of generality, we restrict ourselves to steady heat equation with homogeneous Dirichlet boundary conditions. We refer to Figure \ref{fig:heat_model} for an illustration of this model.

In the spirit of the penalty method, we recast the problem in the following form
\begin{equation}
    \min_{\chi\in\mathcal{U}}\min_{T\in  H_g^1(\Omega)}L(\chi,T) = \int_{\Omega}\kappa(\chi)\nabla T\cdot\nabla T\ dx+\lambda\Big(\dfrac{1}{2}a(\chi;T)-\ell(T)+\dfrac{1}{2}a(\chi;T^*)\Big)+\dfrac{\gamma}{\epsilon}P_{G_\epsilon}(\chi),
\end{equation}
where
\begin{align}
    &a(\chi;T)=\int_\Omega\kappa(\chi)\nabla T\cdot\nabla T\ dx,\quad \ell(T)=\int_\Omega q(\chi)T \ dx,\\
    &T^*\in \arg\min  _{T\in H_g^1(\Omega)}\Big(\dfrac{1}{2}a(\chi;T)-\ell(T)\Big).
\end{align}

Clearly, solving the inner minimization problem is equivalent to solving the following partial differential equation: Find $T\in H_g^1(\Omega)$ such that   
\begin{equation}\label{heat:adjoint}
 \int_\Omega\kappa(\chi)\nabla T\cdot \nabla S\ dx=\dfrac{\lambda}{\lambda+2}\int_\Omega q(\chi)S\ dx\qquad\forall S\in H_D^1(\Omega).
\end{equation}
The Fréchet derivative of the objective functional $L $ with respect to \(\chi\) can easily be  derived as follows.
%\begin{equation}
\begin{align}\label{heat:gradient}
    \dfrac{\partial L}{\partial \chi}(\chi,T)(\hat{\chi})=&\int_\Omega \kappa'(\chi)(\hat{\chi})|\nabla T|^2\ dx+\dfrac{\lambda}{2}\int_\Omega\kappa'(\chi)(\hat{\chi})(|\nabla T|^2-|\nabla T^*|^2)\ dx\nonumber\\
    &-\lambda\int_\Omega q'(\chi)(\hat{\chi})T\ dx+\dfrac{\gamma}{\epsilon}\dfrac{\partial P_{G_\epsilon}}{\partial\chi}(\hat{\chi}).
\end{align}
%\end{equation}
Therefore, the numerical algorithm for the heat transfer problem can be obtained directly.
\begin{algorithm}
\caption{Gradient descent algorithm for the heat transfer problem}
\label{alg:monotonic_descen_with_penlty_heat}
\begin{algorithmic}[1]
\STATE Initialize $k=0,\chi_k$ (The initial value is typically set to a uniform distribution with $ \chi_k = \beta $),  and parameters $\lambda,\gamma,\epsilon,\Delta$.
\WHILE{\( \chi_k \) not converged}
    \STATE Solve \eqref{heat:state} and \eqref{heat:adjoint}.
    \STATE Compute \eqref{heat:gradient}.
    \STATE Update \( \chi_{k}\) using Algorithm \ref{alg:coefficient_line_search}.
    \STATE $k\leftarrow k+1$.
\ENDWHILE
\RETURN Optimal density \( \chi^* \) and function value $L(\chi^*,T^*)$.
\end{algorithmic}
\end{algorithm}

\subsection{Generalized material interpolation function}
Similarly to the variable substitution introduced for the compliant mechanism problem, we can adopt an analogous transformation for the heat transfer problem. Specifically, we define
\begin{equation}
    \Lambda = \Big[ \left( \frac{1}{\kappa_1} - \frac{1}{\kappa_2} \right) \chi + \frac{1}{\kappa_2} \Big]^{-1} \nabla T,
\end{equation}
which plays the role of a flux-like variable analogous to the stress field $\sigma$ in the mechanical setting.

By shifting focus from a purely algebraic or structural perspective to the direct influence on the state equation, we observe that this substitution effectively convexifies the originally semi-convex functional with respect to $\chi$. This convexification is analogous to the relation observed for the compliant mechanism problem:
\begin{align}
    &A(\chi) = (E_{\max} - E_{\min})\chi + E_{\min} 
    \quad \Longleftrightarrow \quad 
    A(\chi) = \Big[ \Big( \frac{1}{E_{\max}} - \frac{1}{E_{\min}} \Big) \chi + \frac{1}{E_{\min}} \Big]^{-1}\quad \chi\in\mathcal{U}, \label{stress_interpolation}\\
    &\kappa(\chi) = (\kappa_1 - \kappa_2)\chi + \kappa_2 
    \quad \Longleftrightarrow \quad 
    \kappa(\chi) = \Big[ \Big( \frac{1}{\kappa_1} - \frac{1}{\kappa_2} \Big) \chi + \frac{1}{\kappa_2} \Big]^{-1}\quad \chi\in\mathcal{U}.\label{conductivity_interpolation}
\end{align}
The substitution thus provides a continuous bridge between the temperature gradient field $\nabla T$ and the flux field $\Lambda$, enabling a more convex and numerically stable optimization landscape.

Within the penalty framework, we further observe that a family of functions can be introduced to control the connectivity of the optimized domain in a manner distinct from conventional filtering or density penalization techniques. We define the following \textbf{generalized material interpolation function} (GMIF):
\begin{equation}
    Y(k_1, k_2, p, \chi) = \left[ (k_1^p - k_2^p) \chi + k_2^p \right]^{1/p}\quad \chi\in\bar{\mathcal{U}},
\end{equation}
where $k_1$ and $k_2$ represent physical quantities appearing in the state equation (such as conductivities or stiffnesses), $p \in \mathbb{R}$ is a tunable parameter, and $\chi \in \bar{\mathcal{U}}$. $p\in [-1,1]$ is of particular interest, as it recovers the left definition of \eqref{conductivity_interpolation} when $p=1$ and the right definition when $p=-1$.

\begin{figure}[htbp]
    \centering
    \includegraphics[width=0.5\linewidth]{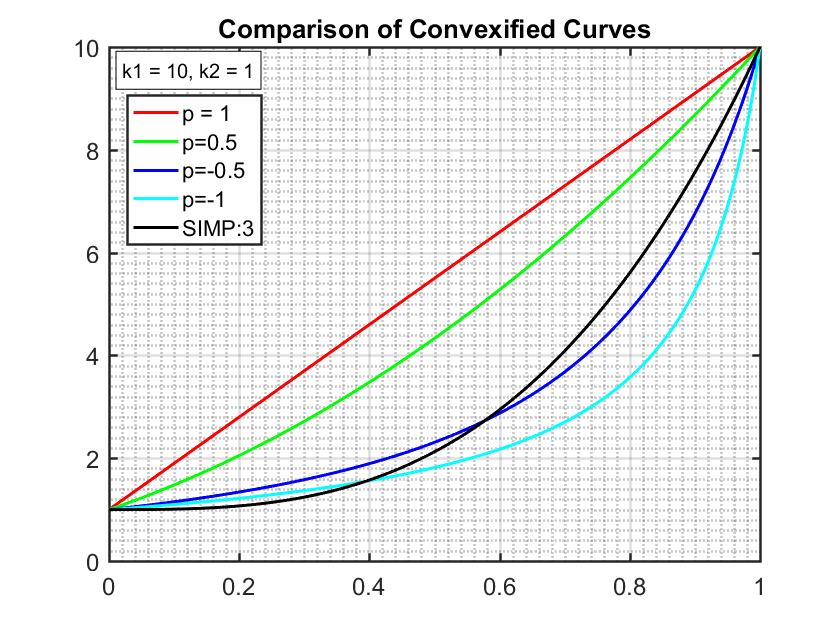}
    \caption{Behavior of the generalized material interpolation function $Y(k_1, k_2, p, \chi)$ for different values of the exponent $p$.}
    \label{fig:gmif-curves}
\end{figure}

As illustrated in Figure~\ref{fig:gmif-curves}, when $p = 1$, $Y$ recovers the standard arithmetic (weighted) average, corresponding to the original state equation. When $p \to -1$, $Y$ approaches the harmonic mean, which aligns with the dual (reciprocal) formulation of the state equation. Numerical experiments reveal that values of $p$ close to $1$ tend to produce more abrupt changes in material distribution and weaker domain connectivity, whereas values close to $-1$ promote smoother transitions and stronger connectivity. This family of interpolations thus provides a flexible and physically meaningful mechanism for controlling topological features in the optimized design.

\section{Numerical results}\label{sec:numerics}
In this section, we present a comprehensive set of numerical experiments to demonstrate the effectiveness, robustness, and practical performance of the proposed penalty-based method. The experiments are organized into two main parts: compliant mechanism design and heat transfer optimization with design-dependent heat sources. In the compliant mechanism part, we focus on the solution quality for two representative benchmark problems. In the heat transfer part, we not only evaluate the overall performance of the method but also investigate the influence of the exponent $p$ of the generalized material interpolation function (GMIF) on optimized topologies, particularly its effect on domain connectivity and material distribution smoothness.

\subsection{Compliant Mechanism Problems}
\begin{figure}[htbp]
\centering
\begin{minipage}[t]{0.48\textwidth}
\centering
\includegraphics[width=0.5\linewidth]{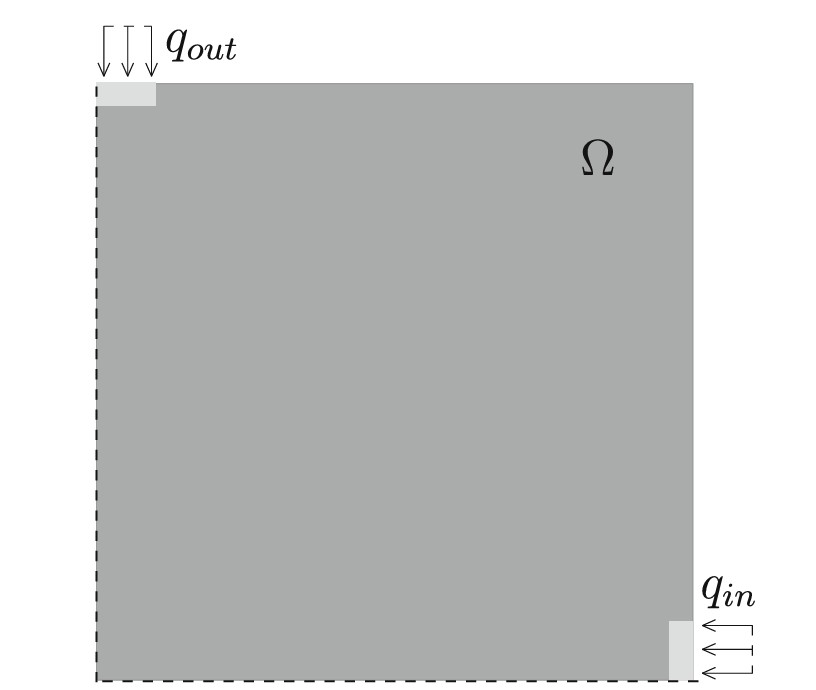}
\label{fig:paper2}
\end{minipage}
\begin{minipage}[t]{0.48\textwidth}
\centering
\includegraphics[width=0.85\linewidth]{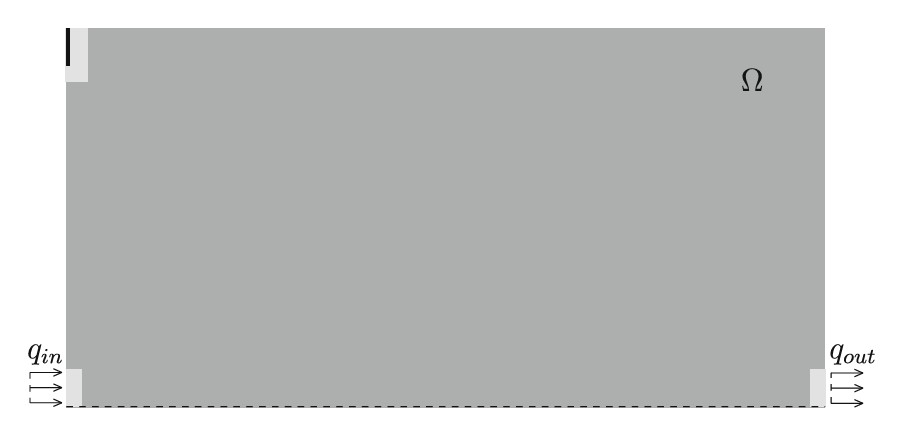}
\label{fig:paper4}
\end{minipage}
\caption{(adapted from \cite{Lops2016}) Dashed lines denote homogeneous tangential Neumann boundaries, thick solid lines indicate homogeneous Dirichlet boundaries, $q_{\text{in}}$ and $q_{\text{out}}$ represent inhomogeneous Neumann boundaries, remaining parts are homogeneous Neumann boundaries. }
\label{fig:models}
\end{figure}

We consider two representative model problems in compliant mechanism design with appropriate boundary conditions (cf. Figure~\ref{fig:models}). The common parameters are as follows: Young's modulus of the solid material $E_{\max} = 5000 \times 8/3$, Young's modulus of the artificial (void) material $E_{\min} = 10^{-5} E_{\max}$, the perimeter penalty parameter $\gamma = 0.1$, and the convolution parameter $\epsilon = h$ (where $h$ is the mesh size). The initial design is uniformly set to $\chi = \beta$ in both cases. For Model Problem 1, the input load is $q_{\text{in}} = -2$ and the output load $q_{\text{out}} = -1$, with a $400 \times 400$ mesh. For Model Problem 2, $q_{\text{in}} = 1$, $q_{\text{out}} = 1$, and a $600 \times 300$ mesh is used.

%The boundary conditions are illustrated in Figures~\ref{fig:models}: dashed lines denote homogeneous tangential Neumann boundaries, thick solid lines represent homogeneous Dirichlet boundaries, $q_{\text{in}}$ and $q_{\text{out}}$ indicate inhomogeneous Neumann boundaries, and all remaining portions are subject to homogeneous Neumann conditions.

\begin{figure}[htbp]
\centering
\begin{minipage}[t]{0.48\textwidth}
    \centering
    \includegraphics[width=0.6\linewidth]{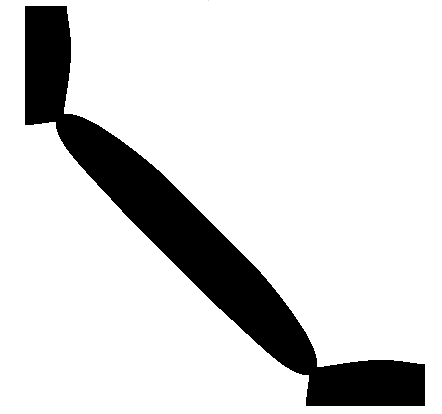}  % 请替换为实际文件名
    \caption{The final optimal shape.}
    \label{fig:final_shape1}
\end{minipage}
\hfill
\begin{minipage}[t]{0.48\textwidth}
    \centering
    \includegraphics[width=0.8\linewidth]{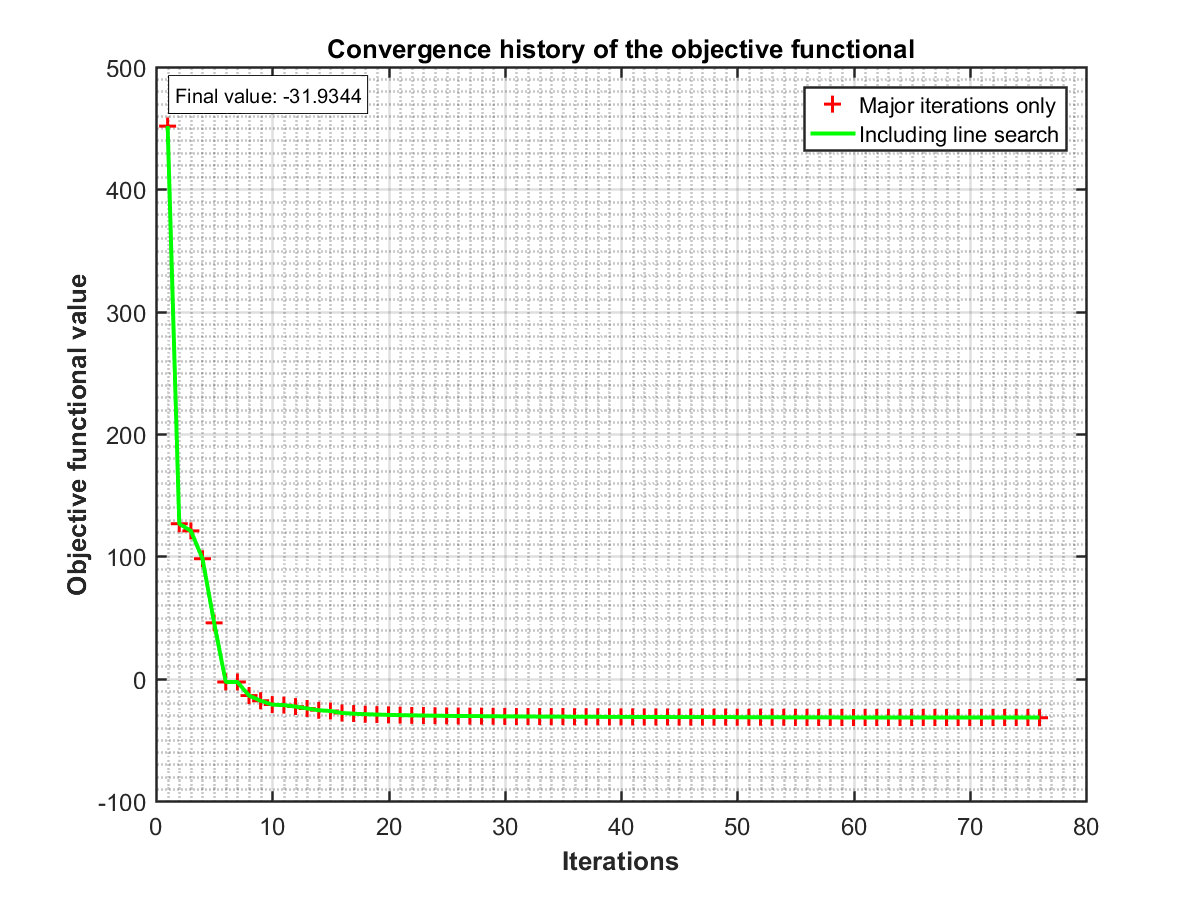}
    \caption{Convergence history of the objective functional.}
    \label{fig:energy1}
\end{minipage}
\label{fig:compliant_results1}
\end{figure}

\begin{figure}[htbp]
\centering
\begin{minipage}[t]{0.48\textwidth}
    \centering
    \includegraphics[width=\linewidth]{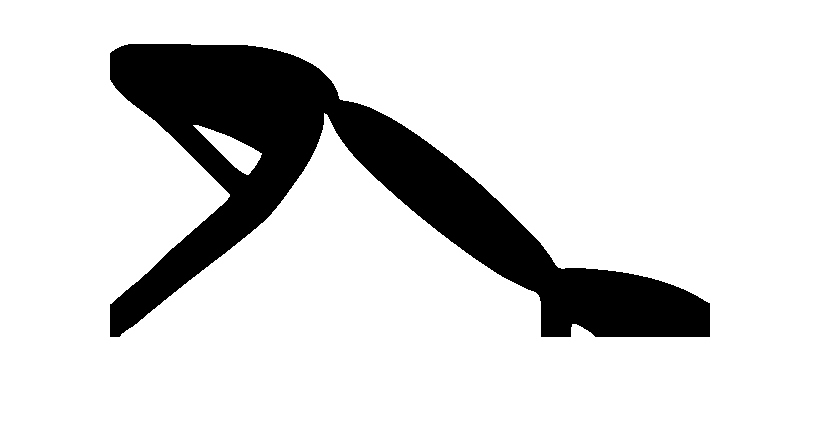}  % 请替换为实际文件名
    \caption{The final optimal shape.}
    \label{fig:final_shape2}
\end{minipage}
\hfill
\begin{minipage}[t]{0.48\textwidth}
    \centering
    \includegraphics[width=0.8\linewidth]{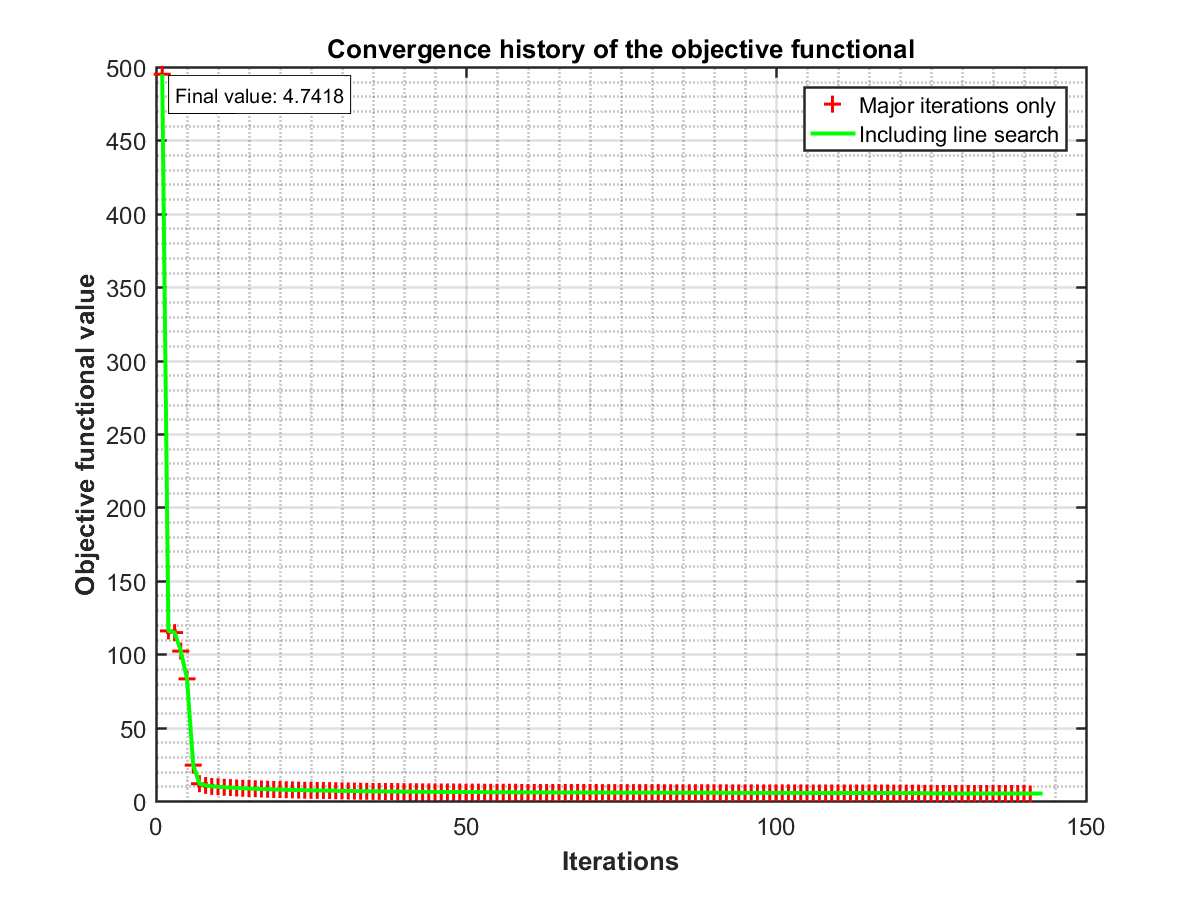}
    \caption{Convergence history of the objective functional.}
    \label{fig:energy2}
\end{minipage}
\label{fig:compliant_results2}
\end{figure}

\begin{figure}[htbp]
\centering
\begin{minipage}[t]{0.48\textwidth}
    \centering
    \includegraphics[width=0.6\linewidth]{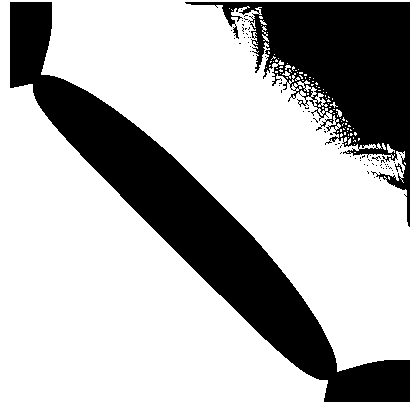}  % 请替换为实际文件名
    \caption{The final optimal shape.}
    \label{fig:final_shape3}
\end{minipage}
\hfill
\begin{minipage}[t]{0.48\textwidth}
    \centering
    \includegraphics[width=0.8\linewidth]{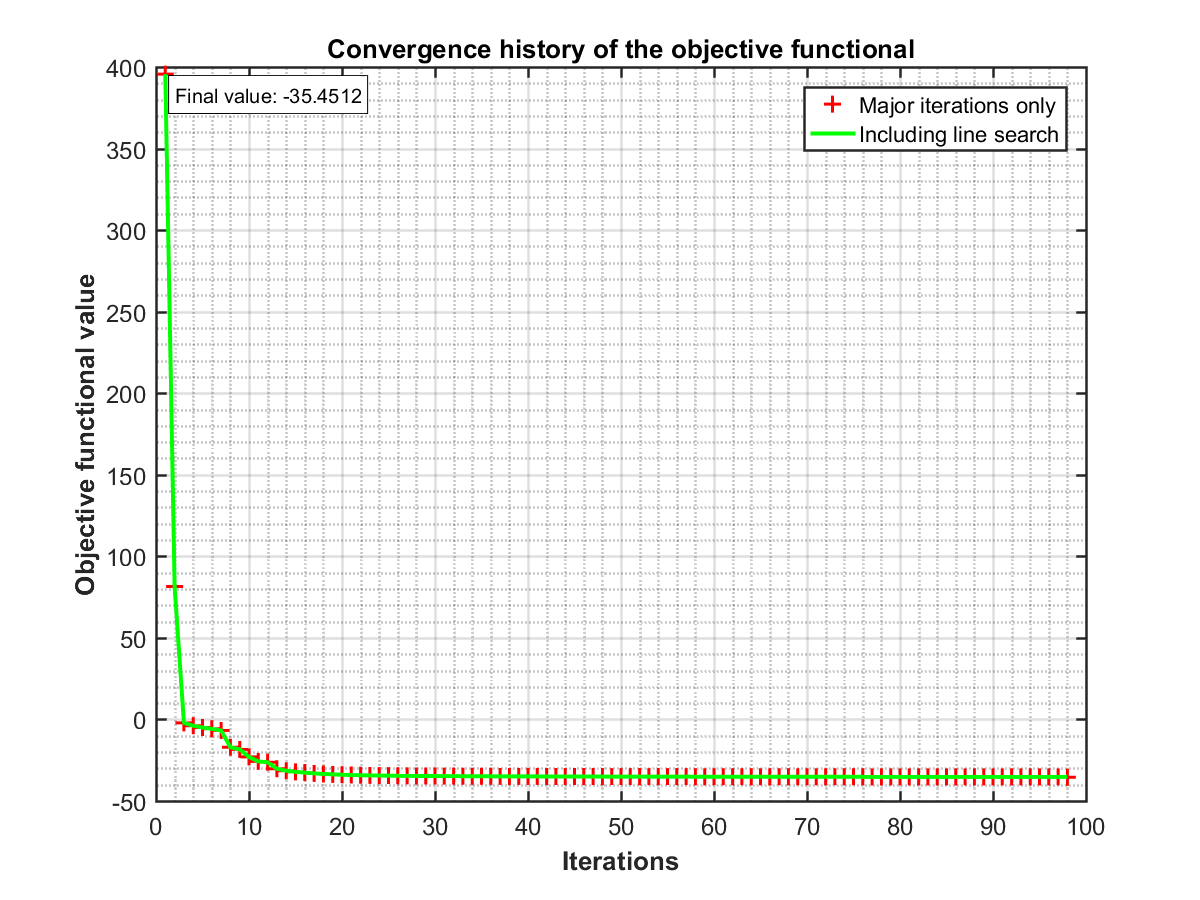}
    \caption{Convergence history of the objective functional.}
    \label{fig:energy3}
\end{minipage}
\label{fig:compliant_results3}
\end{figure}
The results in Figures~\ref{fig:final_shape1}--\ref{fig:energy2} demonstrate an excellent monotonic descent of the objective functional throughout the optimization, with convergence achieved in a relatively small number of iterations. The overall shape of the compliant mechanism emerges rapidly within the first few steps, while subsequent iterations primarily refine fine-scale details.

In Figure~\ref{fig:final_shape1}--\ref{fig:energy3}, the solid green line represents the evolution of the objective functional by also counting the line search steps, while the red plus markers show the evolution with only gradient descent steps. The near-perfect overlap of the two curves experimentally confirms that, although the penalty method incorporates line search to ensure a strict decrease, line search steps rarely occur in practice for compliant mechanism problems.

Additionally, Figure~\ref{fig:final_shape3} shows the result obtained for Model Problem 1 when the volume inequality constraint is replaced by an equality constraint. The method still produces a high-quality shape, further illustrating its robustness and numerical stability across variations in constraints, mesh resolutions, and loading conditions.

On the other hand, when using the displacement-based formulation \eqref{orign_penalty}, a sufficiently large penalty parameter $  \lambda  $ (in our experiments, $  \lambda = 25  $ is adequate) produces numerical results that are comparable to those obtained with the stress-based approach. This formulation is particularly advantageous when dealing with problems that involve design-dependent body forces.

For Model Problem 2, we also apply the generalized material interpolation function (GMIF) using the sequence of exponents $  p = -0.2, -0.4, -0.6, -0.8, -1  $, where the final optimal configurations and the convergence histories of the objective functional are illustrated in  Figures \ref{fig:ela_gmif_shapes} and \ref{fig:ela_gmif_function}, respectively. When \( p > 0 \), both the stress-based and displacement-based penalty formulations, if initialized with a uniform density \(\chi = \beta\), tend to converge to undesirable local minima. However, meaningful optimal shapes can be readily obtained by using a non-uniform initial configuration.

\begin{figure}[htbp]
\centering
\begin{minipage}[t]{0.19\textwidth}
    \centering
    \includegraphics[width=\linewidth]{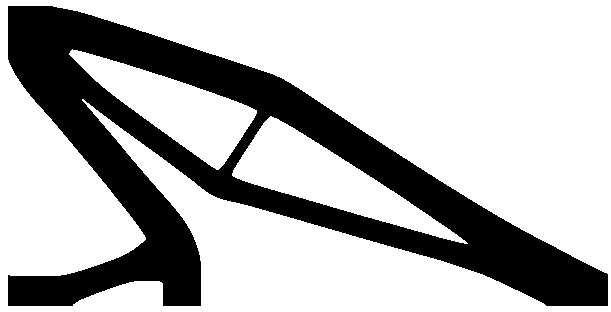}   % p = 1 的最终形状
    \label{fig:ela_p02}
\end{minipage}\hfill
\begin{minipage}[t]{0.19\textwidth}
    \centering
    \includegraphics[width=\linewidth]{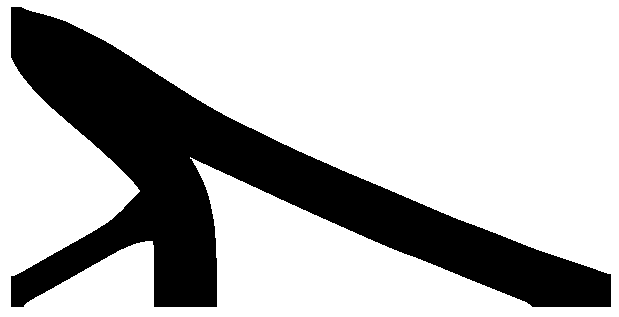}  % p = 0.5
    \label{fig:ela_p04}
\end{minipage}\hfill
\begin{minipage}[t]{0.19\textwidth}
    \centering
    \includegraphics[width=\linewidth]{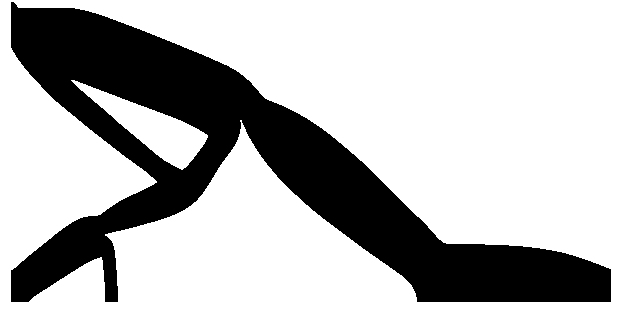}  % p = 0.1
    \label{fig:ela_p06}
\end{minipage}\hfill
\begin{minipage}[t]{0.19\textwidth}
    \centering
    \includegraphics[width=\linewidth]{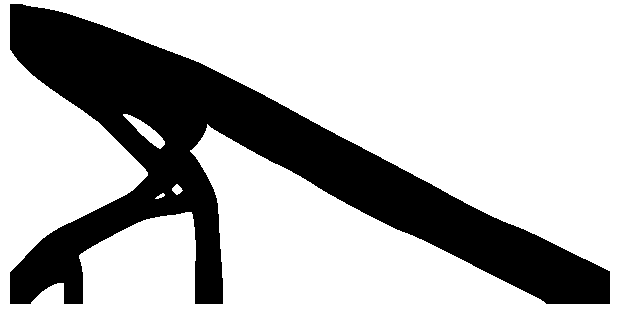} % p = -0.1
    \label{fig:ela_p08}
\end{minipage}\hfill
\begin{minipage}[t]{0.19\textwidth}
    \centering
    \includegraphics[width=\linewidth]{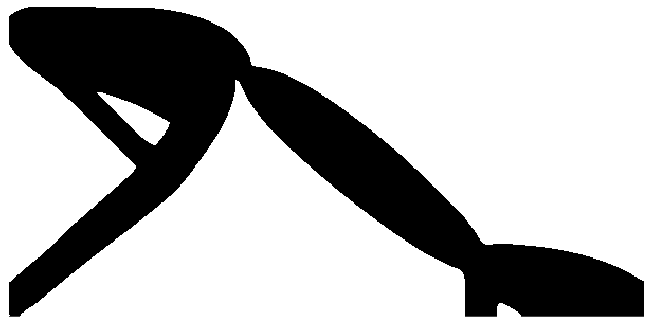}  % p = -1
    \label{fig:ela_pm1}
\end{minipage}
\caption{The final optimal shapes for the compliant mechanism problem obtained with the generalized material interpolation function (GMIF) with different exponent values of $p$. From left to right: $p=-0.2$, $p=-0.4$, $p=-0.6$, $p=-0.8$, $p=-1$.}
\label{fig:ela_gmif_shapes}
\end{figure}

\begin{figure}[htbp]
\centering
\begin{minipage}[t]{0.19\textwidth}
    \centering
    \includegraphics[width=\linewidth]{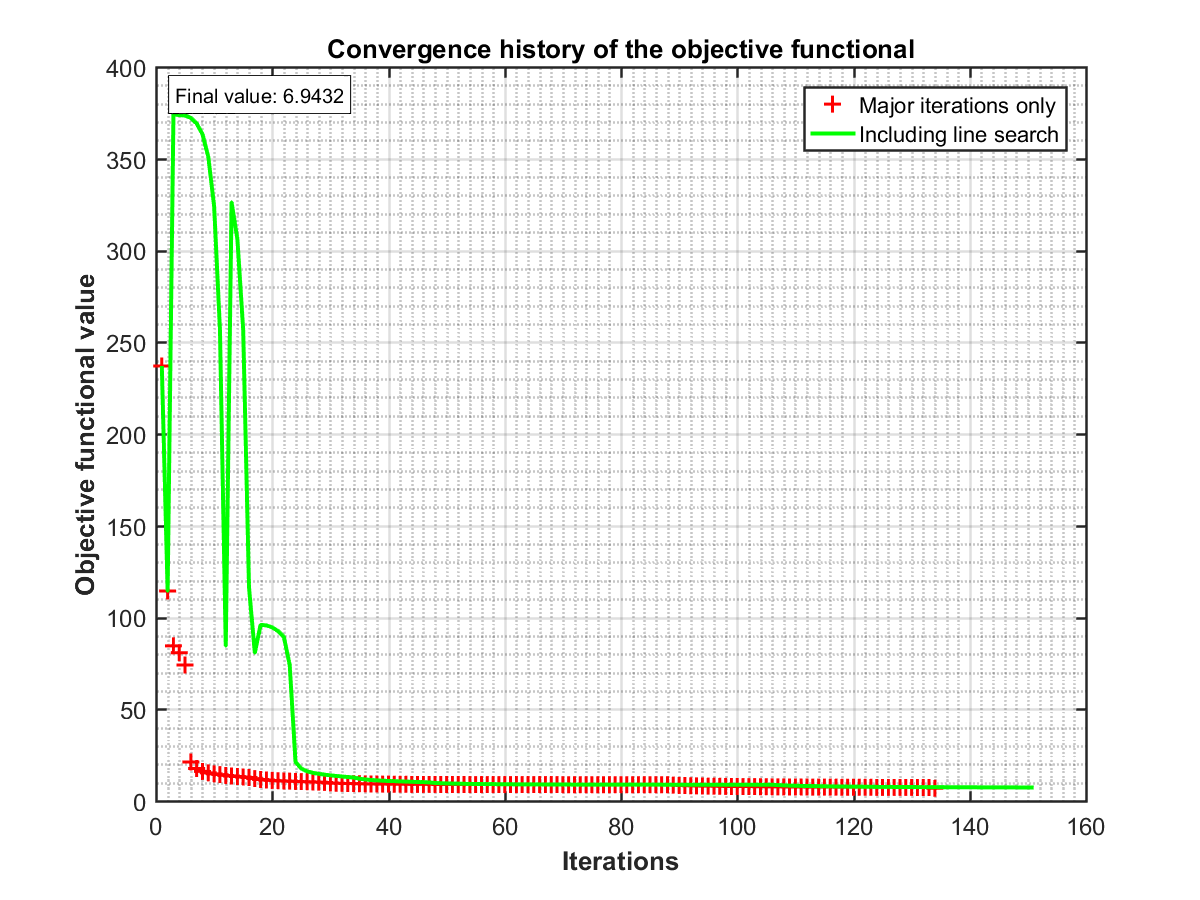}   % p = 1 的最终形状
    \label{fig:ela_pf02}
\end{minipage}\hfill
\begin{minipage}[t]{0.19\textwidth}
    \centering
    \includegraphics[width=\linewidth]{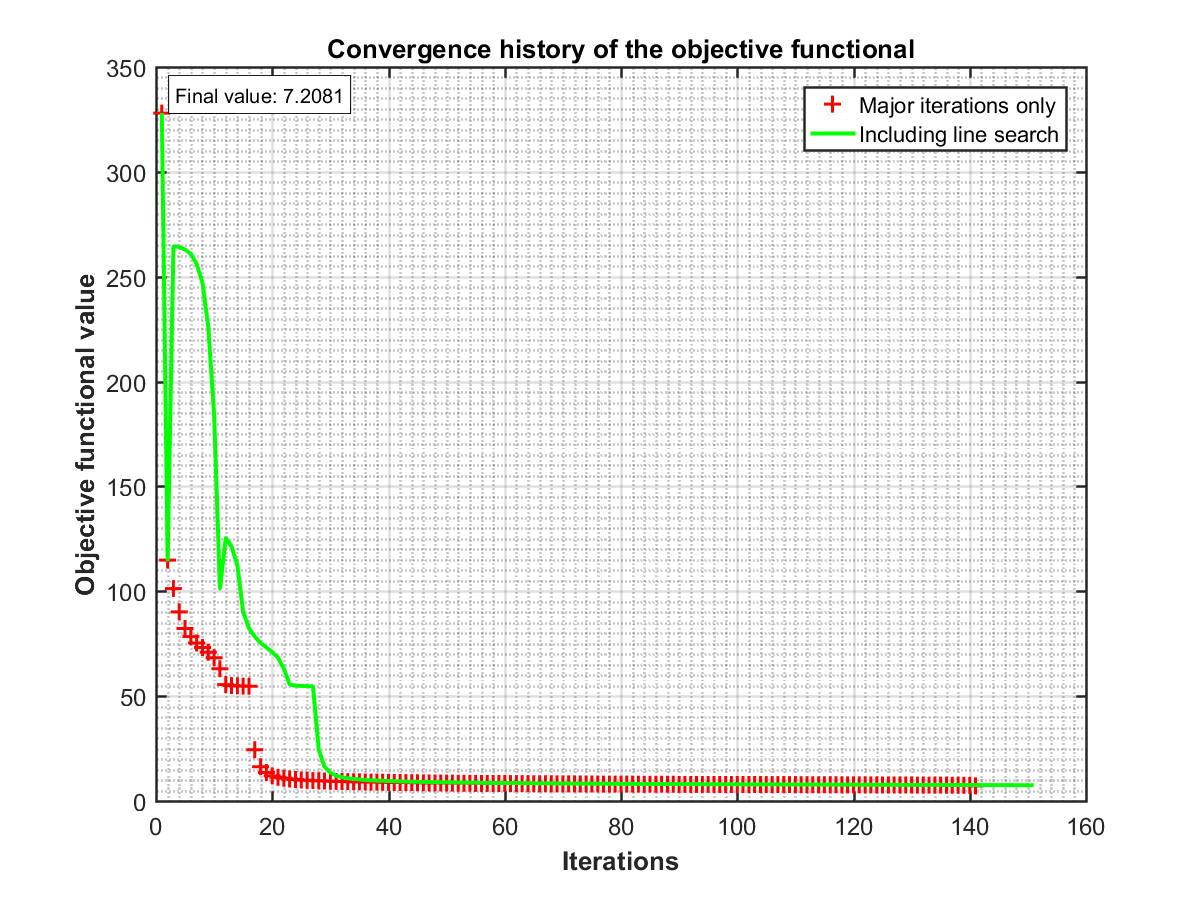}  % p = 0.5
    \label{fig:ela_pf04}
\end{minipage}\hfill
\begin{minipage}[t]{0.19\textwidth}
    \centering
    \includegraphics[width=\linewidth]{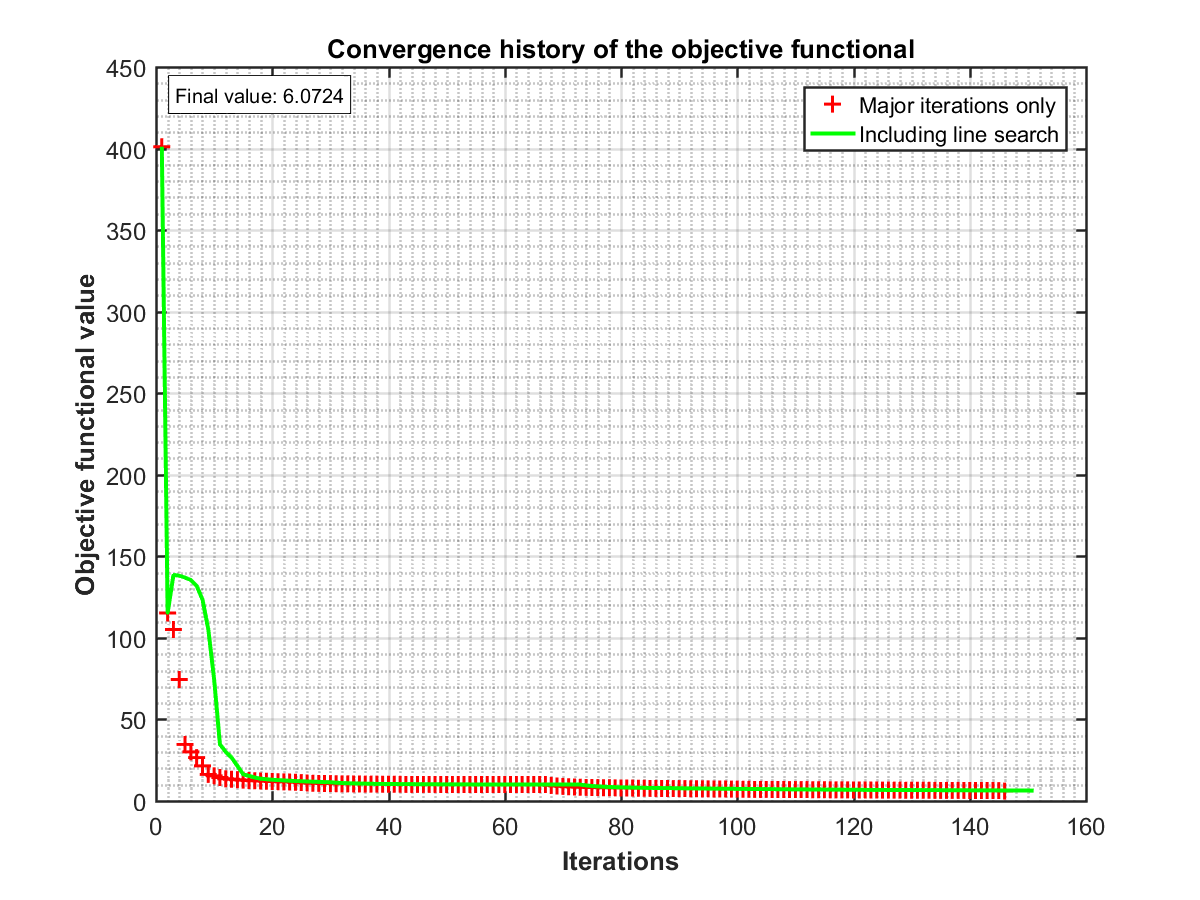}  % p = 0.1
    \label{fig:ela_pf06}
\end{minipage}\hfill
\begin{minipage}[t]{0.19\textwidth}
    \centering
    \includegraphics[width=\linewidth]{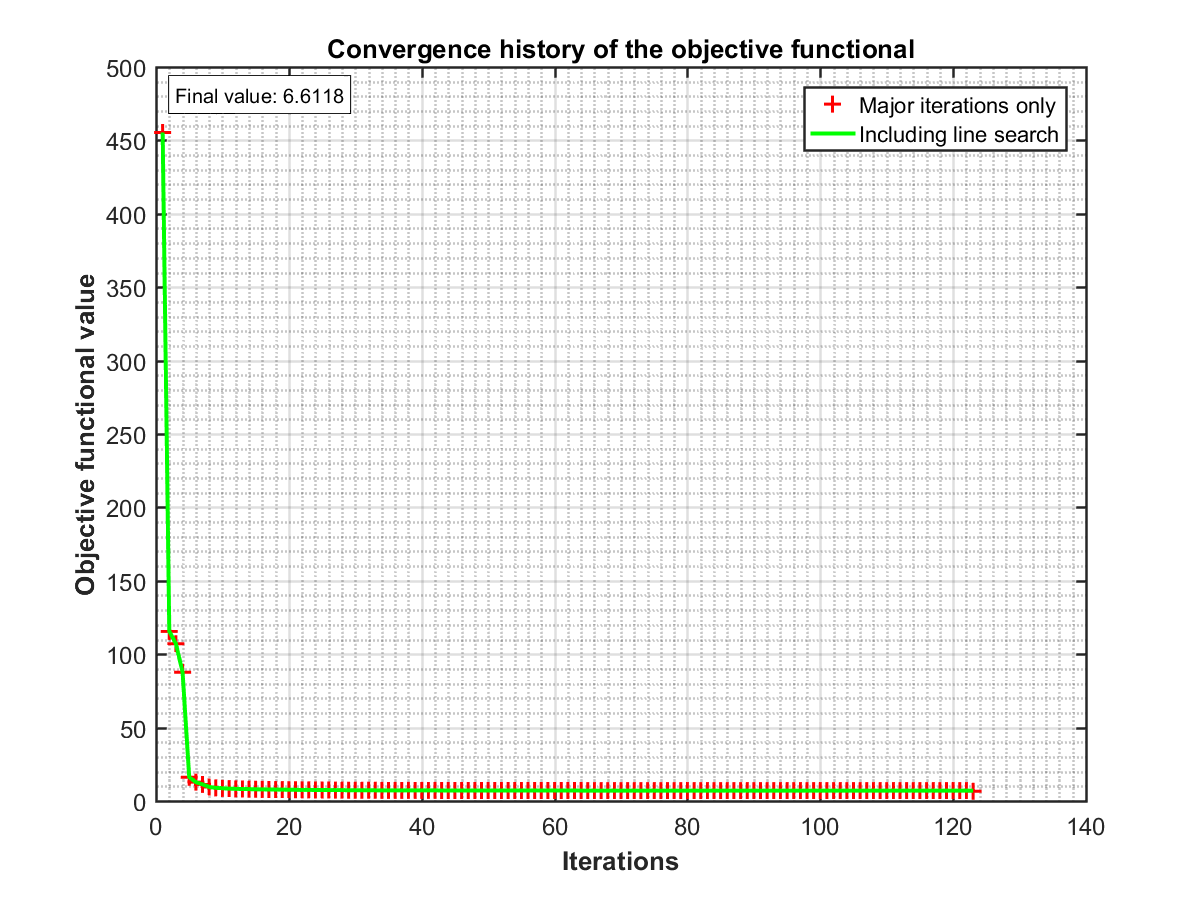} % p = -0.1
    \label{fig:ela_pf08}
\end{minipage}\hfill
\begin{minipage}[t]{0.19\textwidth}
    \centering
    \includegraphics[width=\linewidth]{ela-3-function_val_comp.png}  % p = -1
    \label{fig:ela_pfm1}
\end{minipage}
\caption{The convergence history of the objective functional for compliant mechanism problems obtained with the generalized material interpolation function (GMIF) with different exponent values of $p$. From left to right: $p=-0.2$, $p=-0.4$, $p=-0.6$, $p=-0.8$, $p=-1$.}
\label{fig:ela_gmif_function}
\end{figure}

% \begin{figure}[htbp]
% \centering
% \begin{minipage}[t]{0.48\textwidth}
%     \centering
%     \includegraphics[width=0.8\linewidth]{elastic_shape1_origin.png}  % 请替换为实际文件名
%     \caption{The final optimized shape.}
% \end{minipage}
% \hfill
% \begin{minipage}[t]{0.48\textwidth}
%     \centering
%     \includegraphics[width=0.8\linewidth]{ela-8-function_val_comp.png}
%     \caption{Convergence history of the objective functional.}
% \end{minipage}
% \label{fig:compliant_results4}
% \end{figure}
% \begin{figure}[htbp]
% \centering
% \begin{minipage}[t]{0.48\textwidth}
%     \centering
%     \includegraphics[width=0.8\linewidth]{ela-7-shape.png}  % 请替换为实际文件名
%     \caption{The final optimized shape.}
% \end{minipage}
% \hfill
% \begin{minipage}[t]{0.48\textwidth}
%     \centering
%     \includegraphics[width=0.8\linewidth]{ela-7-function_val_comp.png}
%     \caption{Convergence history of the objective functional.}
% \end{minipage}
% \label{fig:compliant_results4}
% \end{figure}

\subsection{Heat transfer problems}
In this subsection, we consider the model problem for heat transfer optimization, cf. Figure \ref{fig:heat_model}. In the schematic diagram, thin solid lines denote homogeneous Neumann boundary conditions, while thick solid lines indicate inhomogeneous Dirichlet boundary conditions.
\begin{figure}[htbp]
\centering  
\begin{minipage}[t]{0.48\textwidth}
    \centering
    \includegraphics[width=0.6\linewidth]{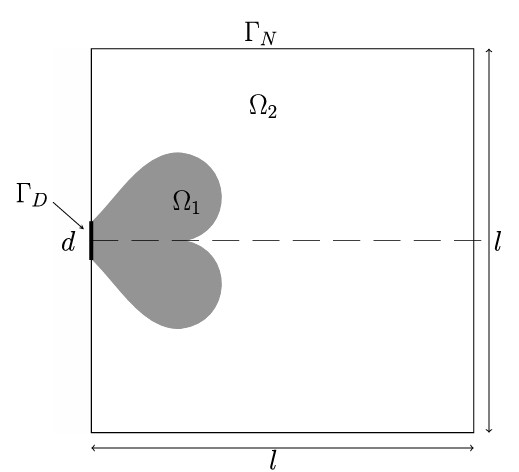}  % 请替换为实际文件名
    \caption{An illustration of the heat transfer model (cf. \cite{CHEN2024113119}).}
    \label{fig:heat_model}
\end{minipage}
\hfill
\begin{minipage}[t]{0.48\textwidth}
    \centering
    \includegraphics[width=0.6\linewidth]{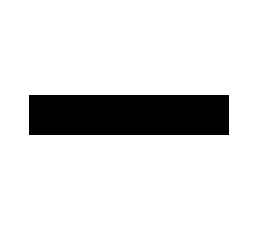}
    \caption{The initial shape.}
    \label{fig:heat_initial_shape}
\end{minipage}
\label{fig:heat_problem}
\end{figure}
In numerical experiments for the heat transfer problem, we fix the following parameters: thermal conductivities $\kappa_1 = 10$ and $\kappa_2 = 1$, heat generation rates $q_1 = 1$ and $q_2 = 100$, a uniform mesh with elements $600 \times 600$, penalty coefficient $\lambda = 0.1$, perimeter penalty coefficient $\gamma = 0.1$ and convolution parameter $\epsilon = h$ (where $h$ is the mesh size). The homogeneous Dirichlet boundary condition is imposed on a segment of length $1/5$ of the domain width, centered on the left boundary. The initial design is a thin horizontal strip of width $1/5$ of the domain width, located in the center of the domain (cf. Figure \ref{fig:heat_initial_shape}).

\begin{figure}[htbp]
\centering
\begin{minipage}[t]{0.19\textwidth}
    \centering
    \includegraphics[width=\linewidth]{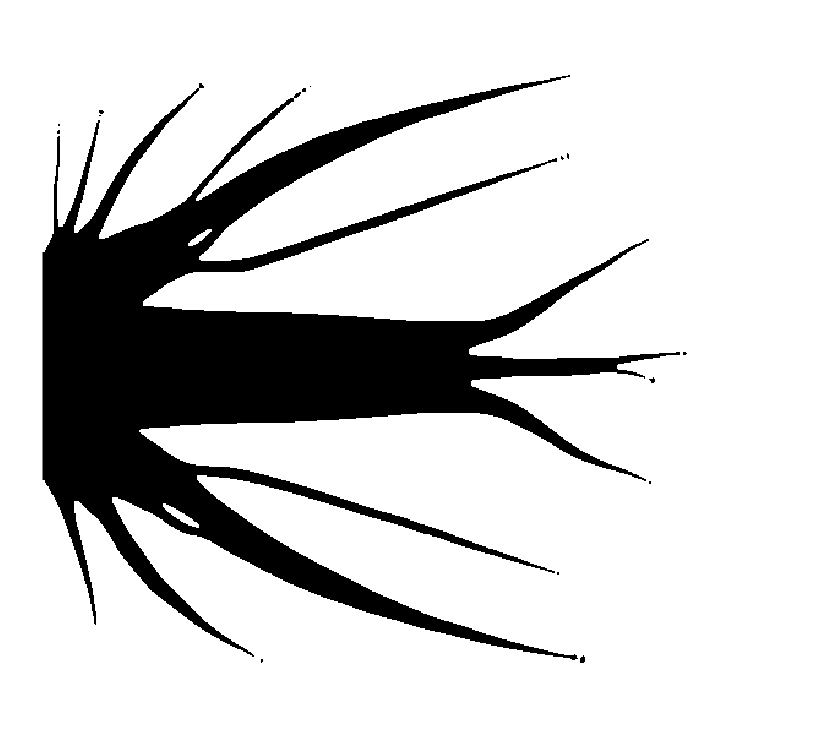}   % p = 1 的最终形状
    \label{fig:heat_p1}
\end{minipage}\hfill
\begin{minipage}[t]{0.19\textwidth}
    \centering
    \includegraphics[width=\linewidth]{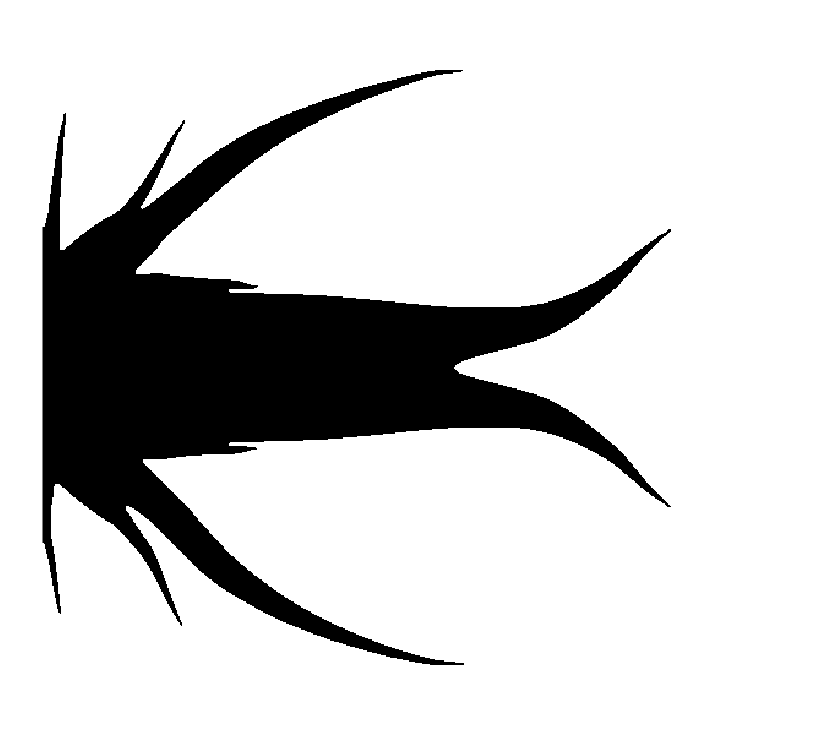}  % p = 0.5
    \label{fig:heat_p05}
\end{minipage}\hfill
\begin{minipage}[t]{0.19\textwidth}
    \centering
    \includegraphics[width=\linewidth]{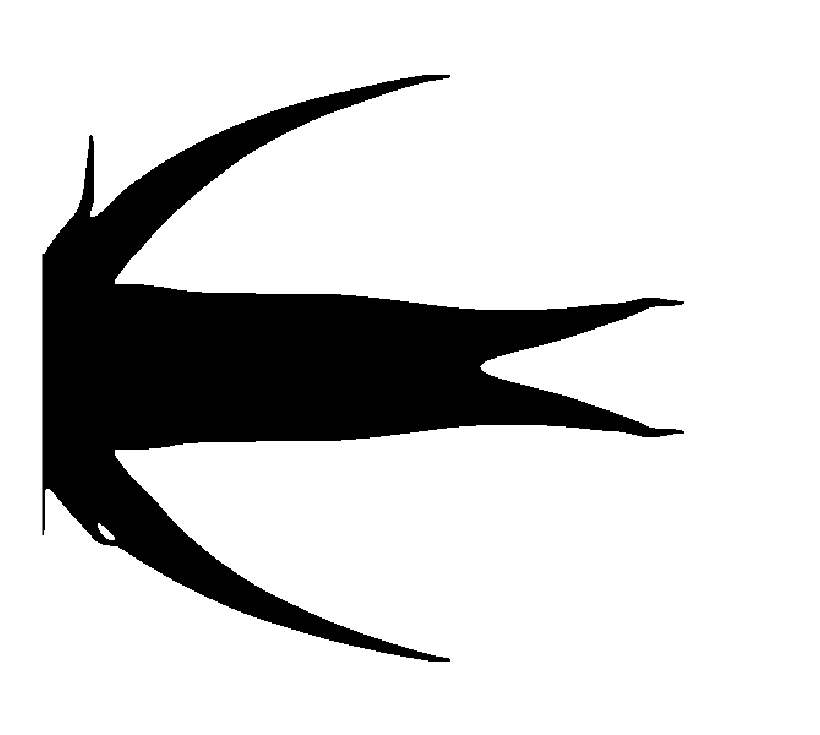}  % p = 0.1
    \label{fig:heat_p01}
\end{minipage}\hfill
\begin{minipage}[t]{0.19\textwidth}
    \centering
    \includegraphics[width=\linewidth]{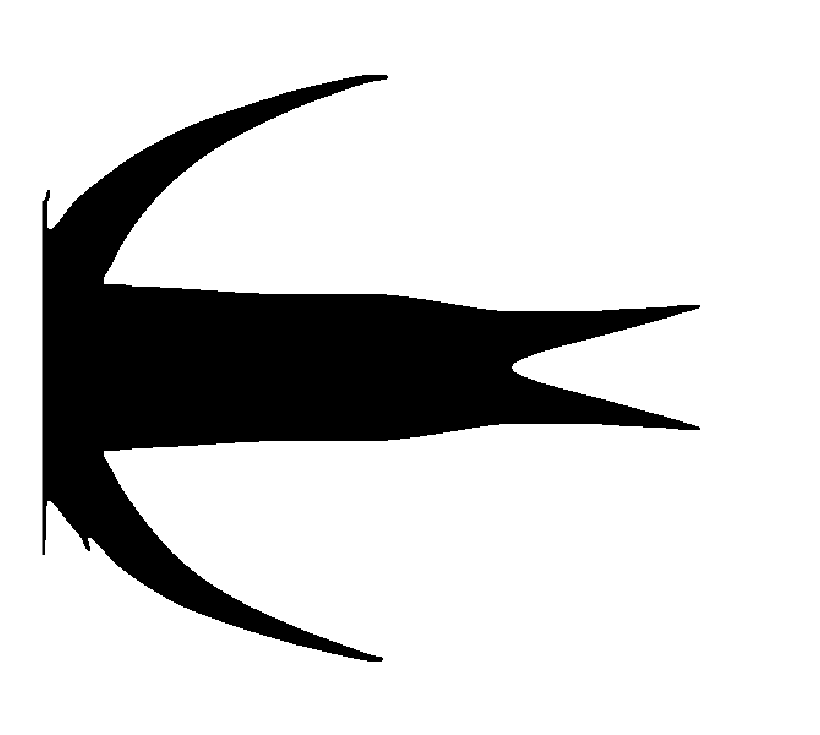} % p = -0.1
    \label{fig:heat_pm01}
\end{minipage}\hfill
\begin{minipage}[t]{0.19\textwidth}
    \centering
    \includegraphics[width=\linewidth]{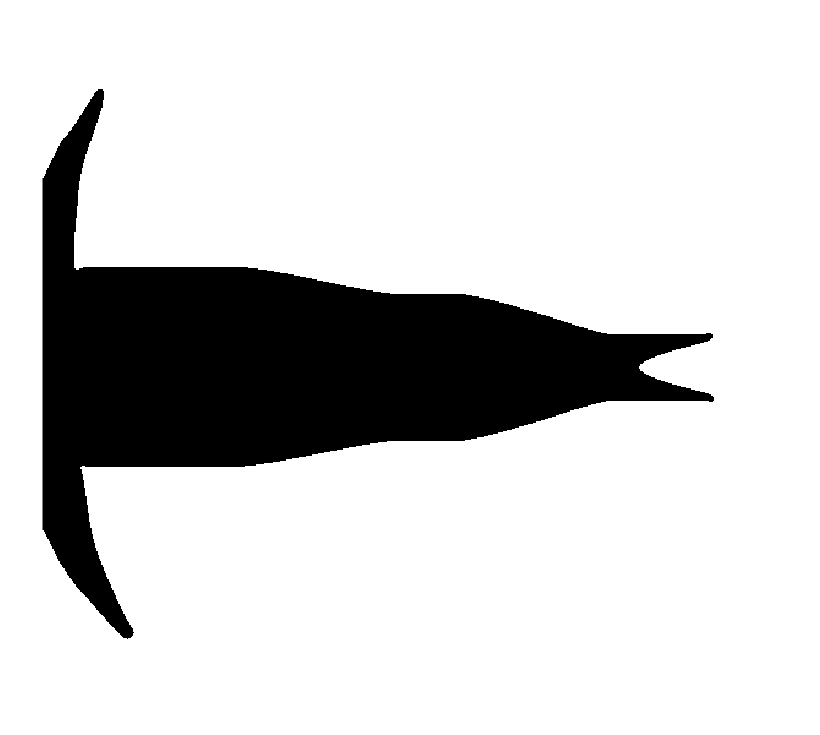}  % p = -1
    \label{fig:heat_pm1}
\end{minipage}
\caption{The final optimal shapes for the heat transfer problem obtained using the generalized material interpolation function (GMIF) with different exponent values $p$. From left to right: $p=1$, $p=0.5$, $p=0.1$, $p=-0.1$, $p=-1$.}
\label{fig:heat_gmif_shapes}
\end{figure}

\begin{figure}[htbp]
\centering
\begin{minipage}[t]{0.19\textwidth}
    \centering
    \includegraphics[width=\linewidth]{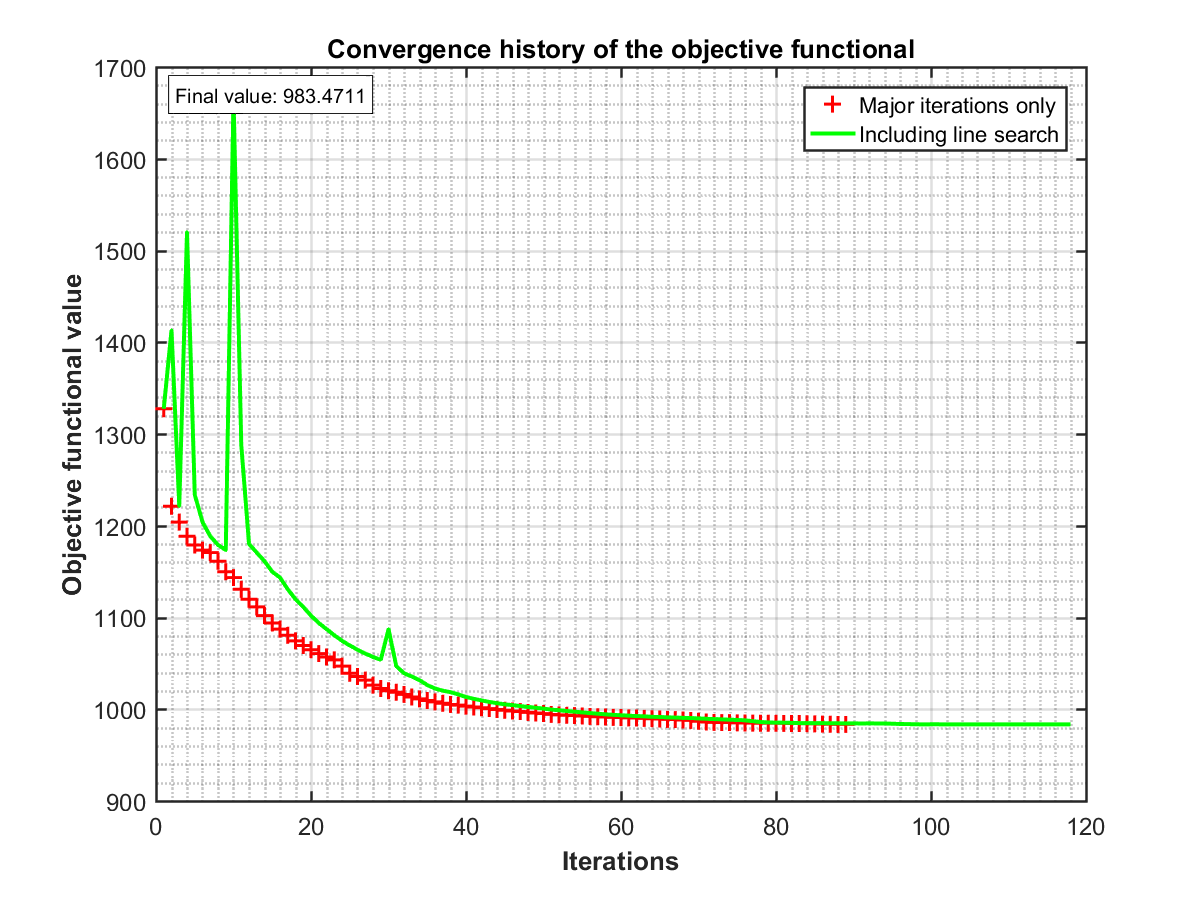}   % p = 1 的最终形状
    \label{fig:heat_p1}
\end{minipage}\hfill
\begin{minipage}[t]{0.19\textwidth}
    \centering
    \includegraphics[width=\linewidth]{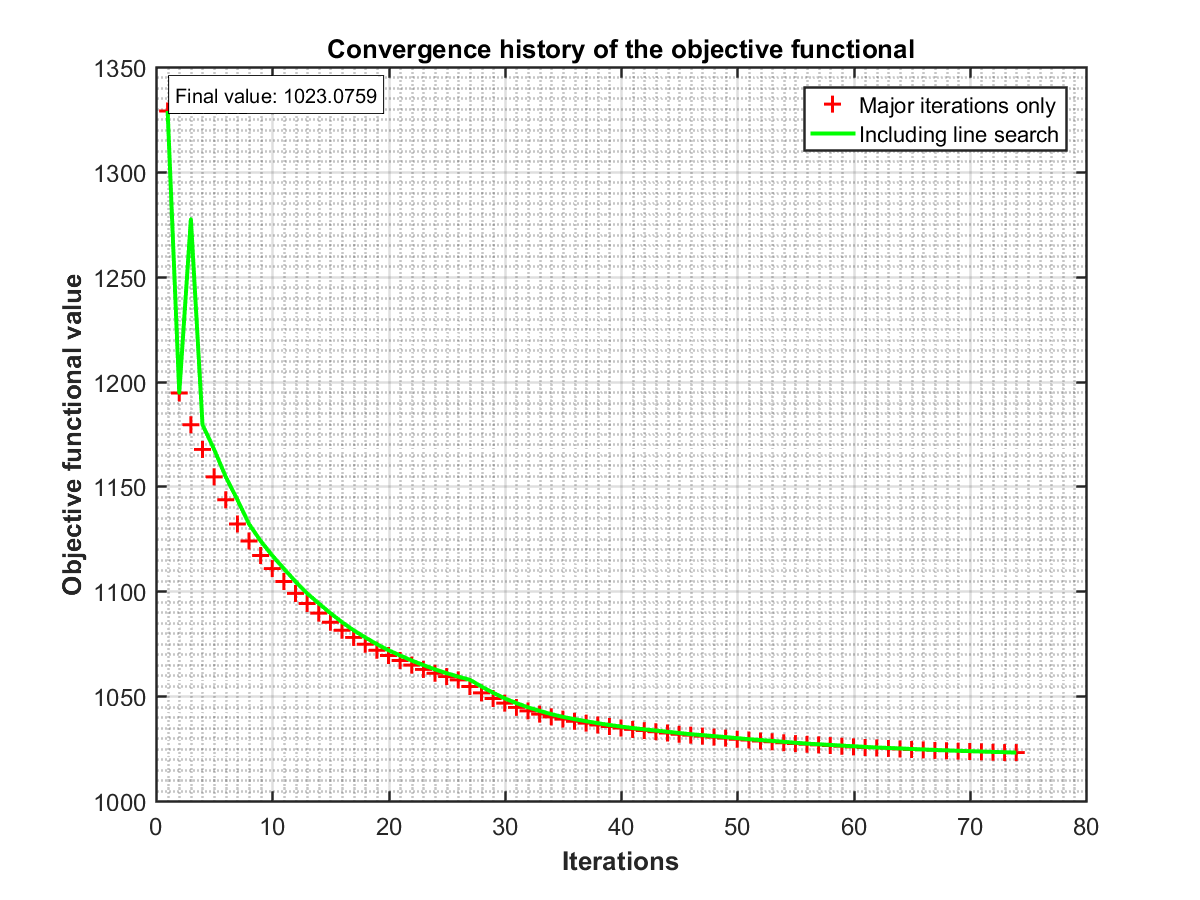}  % p = 0.5
    \label{fig:heat_p05}
\end{minipage}\hfill
\begin{minipage}[t]{0.19\textwidth}
    \centering
    \includegraphics[width=\linewidth]{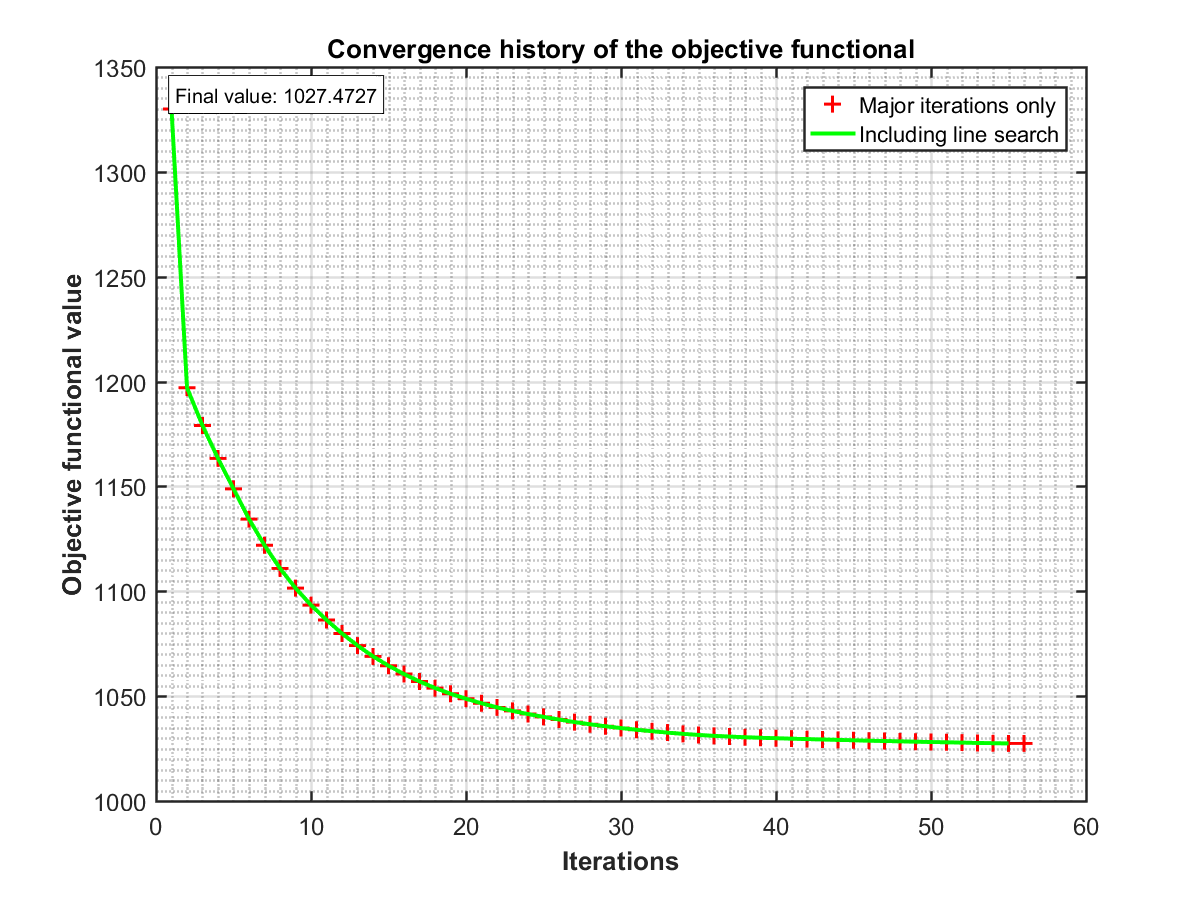}  % p = 0.1
    \label{fig:heat_p01}
\end{minipage}\hfill
\begin{minipage}[t]{0.19\textwidth}
    \centering
    \includegraphics[width=\linewidth]{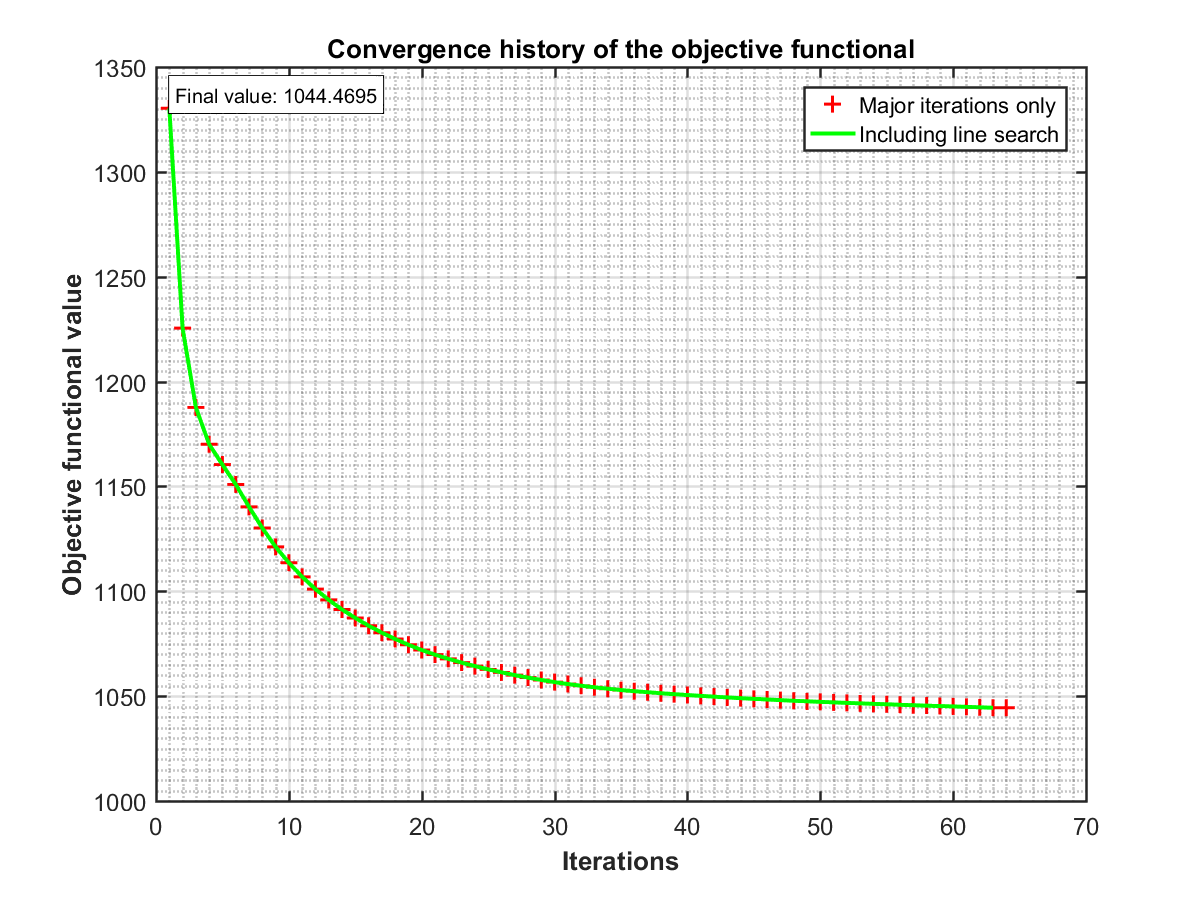} % p = -0.1
    \label{fig:heat_pm01}
\end{minipage}\hfill
\begin{minipage}[t]{0.19\textwidth}
    \centering
    \includegraphics[width=\linewidth]{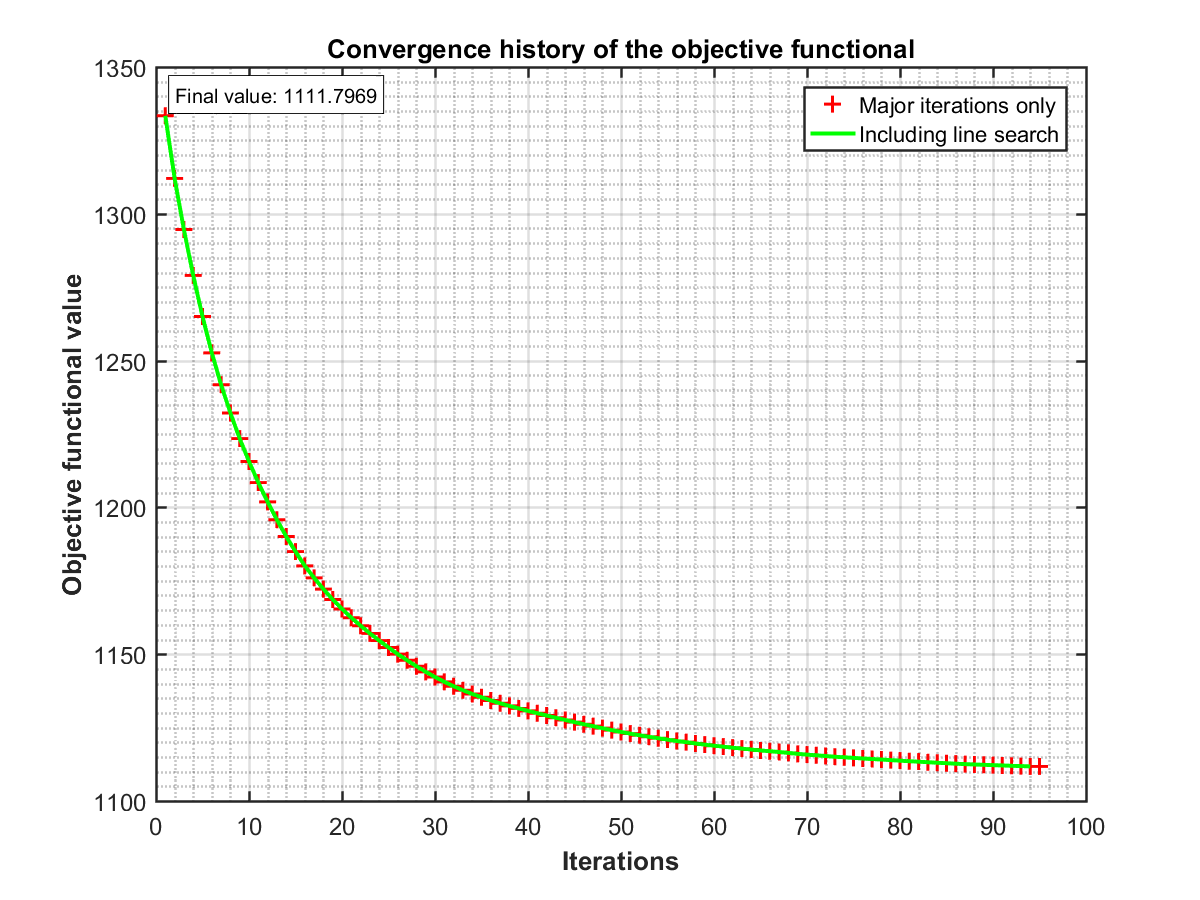}  % p = -1
    \label{fig:heat_pm1}
\end{minipage}
\caption{The convergence history of the objective functional for heat transfer problems obtained using the generalized material interpolation function (GMIF) with different exponent values $p$. From left to right: $p=1$, $p=0.5$, $p=0.1$, $p=-0.1$, $p=-1$.}
\label{fig:heat_gmif_function}
\end{figure}

The numerical results in Figures~\ref{fig:heat_gmif_shapes} and \ref{fig:heat_gmif_function} reveal that, for all tested values of the exponent $p$, the objective functional exhibits a satisfactory monotonic descent throughout the optimization process. However, when $p=1$, a relatively large number of line search procedures are required, and the final optimized shape displays some disconnected features near the tips. These artifacts can be partially mitigated by mesh refinement, although the improvement remains limited. 

Using a sufficiently decreasing $p$ final highly connected shapes can be obtained almost immediately, accompanied by a significant reduction in the number of line search steps. As $p$ decreases further, the frequency of line search evaluations continues to decrease, and when $p=-1$, the algorithm requires no line search steps at all. Nevertheless, the resulting shape at $p=-1$ is suboptimal, suggesting that the optimization process ends at an undesirable local minimum. These observations indicate that the exponent $p$ in the generalized material interpolation function provides an effective and intuitive mechanism to balance topological connectivity, numerical stability, and solution quality.

\section{Conclusion}\label{sec:conclusion}
In this paper, we propose a novel penalty-based reformulation for solving topology optimization problems in non-self-adjoint settings, with a particular focus on compliant mechanism design. Starting from the classical displacement-based formulation, we introduced a variable substitution that transforms the problem into an equivalent bilevel optimization problem expressed in terms of stress-like variables. By applying a carefully designed differentiable penalty term to enforce the state constraint, we obtained a single-level penalized functional that is equivalent to the original problem for sufficiently large penalty parameters. This reformulation not only preserves the essential physics of the problem, but also enables a more stable and convexified optimization landscape.

The proposed method was shown to be versatile and effective in different physics. In the compliant mechanism problem, we rigorously established the equivalence between the penalized formulation and the original problem, and proved that the discrete penalized problems $\Gamma$-converge to the continuous one as the mesh size $h \to 0$ and the regularization parameter $\epsilon \to 0$. Furthermore, under appropriate isolation assumptions, local minimizers of discrete problems converge to isolated local minimizers of the continuous problem. A monotonic descent algorithm was developed that combines gradient descent updates and $L^1$-projection, which guarantees a strict decrease of the objective at each iteration and finite termination in the discrete setting. Numerical experiments confirmed that the method produces high-quality mesh-independent designs with precise control of maximum stress and improved convergence behavior compared to classical approaches.

The same penalty framework was applied directly to the heat transfer problem, demonstrating its generality and robustness in handling design-dependent material properties and source terms. Remarkably, the method retained its effectiveness without requiring major structural modifications.

A key theoretical observation is the emergence of a family of functions, termed the \textbf{Generalized Material Interpolation Function} (GMIF), defined as
\begin{equation}
    Y(k_1,k_2,p,\chi)=\left((k_1^p-k_2^p)\chi+k_2^p\right)^{\frac{1}{p}}.
\end{equation}
By simply tuning the exponent $p$, this family allows continuous interpolation between the arithmetic mean ($p=1$) and the harmonic mean ($p\to -1$), offering a flexible and physically meaningful mechanism to control the connectivity and topological features of optimal design. The numerical results indicate that values of $p$ near $1$ promote abrupt changes and weaker connectivity, while values near $-1$ favor smoother transitions and stronger connectivity. This control mechanism is different from conventional filtering or penalization techniques and provides a promising tool to tailor the topological complexity of the solution.

The contributions of this work lay a solid foundation for further theoretical and numerical developments. Future research directions include: (i) establishing a more abstract mathematical framework for the proposed penalty approach, with a deeper analysis of its functional properties and convergence behavior in general non-self-adjoint settings; (ii) rigorously investigating the mathematical mechanism behind the connectivity control offered by the GMIF family, possibly through variational inequalities or shape calculus; and (iii) extending the method to more complex multiphysics problems, such as fluid-structure interaction or thermo-electromechanical systems, where the interplay between different physical fields poses additional challenges.

Overall, the penalty method and the associated GMIF interpolation introduced in this paper offer a theoretically sound and computationally efficient pathway toward reliable topology optimization in non-self-adjoint problems, with broad potential applications in engineering design and materials science.
%\newpage
\printbibliography 

\end{document}